\documentclass{article}

\usepackage[utf8]{inputenc}	% Einstellung der Eingabekodierung
\usepackage[T1]{fontenc}	% Einstellung Schriftkodierung

\usepackage{amsmath}		
\usepackage{amsthm}
\usepackage{amssymb}
\usepackage{amsfonts}
\usepackage{dsfont}
\usepackage{bm}
\usepackage[%
  inner=25mm,%
  outer=35mm,%
  top=25mm,%
  bottom=30mm,%
  includeheadfoot,%
  headheight=14pt,%
  marginparwidth=25mm%
]{geometry}			% Anpassung Satzspiegel
\usepackage{parskip}		% Kein Absatzeinzug

\usepackage{enumitem}

\usepackage[most]{tcolorbox}

\usepackage{hyperref}
\newcommand{\field}[1]{\ensuremath{\mathbb{#1}}}

\newcommand{\N}{\field N}
\newcommand{\Z}{\field{Z}}

\newcommand{\R}{\field{R}}
\newcommand{\C}{\field{C}}
\newcommand{\K}{\mathbb{K}}

\newcommand{\Ept}{\field E}
\newcommand{\Prob}{\field P}

\newcommand{\bc}{\bm{c}}

\newcommand{\bn}{\bm{n}}
\newcommand{\bv}{\bm{v}}

\newcommand{\bx}{\bm{x}}

\newcommand{\bI}{\bm{I}}
\newcommand{\bA}{\bm{A}}

\newcommand{\bF}{\bm{F}}

\newcommand{\bpsi}{\bm\psi}
\newcommand{\bvarphi}{\bm\varphi}

\newcommand\dint{\,\mathrm{d}}

\newcommand{\cC}{\mathcal{C}}

\newcommand{\cE}{\mathcal{E}}
\newcommand{\cF}{\mathcal{F}}
\newcommand{\cH}{\mathcal{H}}
\newcommand{\cK}{\mathcal{K}}
\newcommand{\cL}{\mathcal{L}}

\newcommand{\cN}{\mathcal{N}}
\newcommand{\cP}{\mathcal{P}}
\newcommand{\cQ}{\mathcal{Q}}

\newcommand{\Id}{\ensuremath{\mathrm{Id}}}

\newcommand{\kerplus}{\ensuremath{\ker_{\ge0}}}
\newcommand{\rank}{\ensuremath{\mathrm{rank}}}
\newcommand{\supp}{{\rm  supp \, }}

\newcommand{\ii}{\mathrm{i}}

\DeclareMathOperator*{\argmax}{\mathrm{arg\,max}}

\DeclareMathOperator{\effdim}{eff-dim}
\DeclareMathOperator{\Tr}{tr}
\DeclareMathOperator{\conv}{conv}
\DeclareMathOperator{\cone}{cone}
\DeclareMathOperator{\lin}{lin}

\DeclareMathOperator{\intconv}{\textstyle{\int}-conv}
\DeclareMathOperator{\intcone}{\textstyle{\int}-cone}
\DeclareMathOperator{\intlin}{\textstyle{\int}-lin}
\DeclareMathOperator{\sgn}{sgn}

\renewcommand{\Im}{\operatorname{Im}}
\renewcommand{\Re}{\operatorname{Re}}

\theoremstyle{plain}
\newtheorem{theorem}{Theorem}[section]
\newtheorem{lemma}[theorem]{Lemma}
\newtheorem{corollary}[theorem]{Corollary}
\newtheorem{proposition}[theorem]{Proposition}

\theoremstyle{definition}
\newtheorem{definition}[theorem]{Definition}
\newtheorem{remark}[theorem]{Remark}
\newtheorem{example}[theorem]{Example}

\numberwithin{equation}{section}

\title{Beyond Tchakaloff Quadrature: Positive Functionals,\\ Frames and Widths}

\author{Martin Sch\"afer\footnote{Corresponding author. Email: martin.schaefer@mathematik.tu-chemnitz.de}
\and Tino Ullrich}
\date{Chemnitz Technical University, 09111 Chemnitz, Germany \\[2ex]
\today\\[3ex]
\small{Dedicated to the memory of Peter G.~Casazza}}

\begin{document}

\maketitle

\begin{abstract}
\noindent
Tchakaloff's theorem from 1957 asserts
the existence of exact quadrature rules
with non-negative weights for any polynomial space of finite degree on $\R^d$ if the underlying measure is positive, compactly supported, 
and absolutely continuous 
with respect to the Lebesgue measure. This classical result coined the term Tchakaloff quadrature
for quadrature that is exact and only uses non-negative weights. It has been a long-standing endeavor, 
under which conditions
such rules exist.
A final answer was given in 2012 by 
Bisgaard 
with the insight that, in fact, every finite-dimensional space of integrable functions on a positive measure space admits them. 
In this article we recall 
this result 
and provide a major extension to the question of positive discretizability of $\C$-linear functionals on finite-dimensional spaces. We introduce the notion of strict $S$-positivity for such functionals, where $S$ are subsets of the functional's domain, and show the equivalence of positive discretizability to being strictly $S$-positive for a suitable choice of $S$.
We further investigate consequences for other 
discretization problems. 
One fundamental 
implication is
the guaranteed existence of $L_p$-Marcinkiewicz-Zygmund equalities 
in finite-dimensional spaces of $p$-integrable functions in case that $p$ is an even integer,
another 
the exact discretizability of any frame in $\K^n$, where $\K\in\{\R,\C\}$,
if a rescaling of the frame elements is allowed.
In addition, we provide bounds for Tchakaloff quadrature widths $\kappa_n^+$ and, addressing the question of 
constructibility
of discretization points, establish a connection  
to $D$-optimal design.
\end{abstract}


%%%%%%%%%%%%%%%%%%%%%%%%%%%%%%%%%%%%%%%%%%%%%%%%%%%%
%--------------------------------------------------%
%%%%%%%%%%%%%%%%%%%%%%%%%%%%%%%%%%%%%%%%%%%%%%%%%%%%

\section{Introduction}

The existence of exact quadrature rules for a given finite-dimensional function space is an important question 
from a theoretical as well as from a practical point of view.  
When weights are allowed in the quadrature and when there is enough freedom in the choice of them,
it turns out to be an easily solvable question, see~Theorems~\ref{thm:quad1} and~\ref{thm:quad2} in this article. It becomes more difficult
if one imposes conditions on the weights and strives, for example, for equally weighted quadrature or rules with non-negative weights. We refer to~\cite{NoWoII,DiKuSl13} for some general background.
Exact quadrature based on non-negative weights is called Tchakaloff quadrature and was first analyzed in the midst of the last century. The research was triggered in 1957 by Tchakaloff's celebrated theorem~\cite{Tschak57} proving the existence of such rules in
polynomial spaces on $\R^d$. Other notable researchers of that early time include Davis~\cite{Davis1967} and Wilson~\cite{Wil1968,Wil1969}.
Since then, it has been a long-time 
endeavor to identify conditions 
on function spaces 
that guarantee the existence of Tchakaloff
quadrature rules.

After many intermediate advances,
requiring for example
compactness assumptions on the domain and 
continuous functions 
or certain moment conditions, see~e.g.~Mysovskikh~\cite{MYSOVSKIKH1975221}, Putinar~\cite{Pu97}, or Curto and Fialkow~\cite{CurFial2002}, 
Bayer and Teichmann~\cite{BayTeich06} in 2006 were the first to provide a general result for polynomial spaces on $\R^d$.  
A complete answer for general function spaces on general domains was given in 2012 by Bisgaard~\cite[Thm.~5.1]{BerSas12}, somewhat hidden in an article authored by Berschneider and Sasv\'{a}ri.
As it turned out, 
restrictions on the function space are not necessary
at all and Tchakaloff rules hold in full generality for positive measures in function spaces of finite dimension.
This -- maybe surprising -- fact may not be so well-known outside the expert community,
as it was e.g.\ not to the authors of this article.
Up to now several newer versions of Tchakaloff results, yet not as general as the one  from~\cite{BerSas12}, have been published,~\cite[Thm.\ 4.1]{DaPrTeTi19} 
and~\cite[Thm.\ 1]{migliorati_stable_2022} 
handling the case of real-valued functions
and~\cite[Thm.~4.3]{BKPSU24} handling complex-valued functions. 

The research in the last mentioned article~\cite{BKPSU24}, which we co-authored, draws a close connection of Tchakaloff quadrature to frame-theoretic questions and $D$-optimal design and was actually the trigger for the investigations leading to this article. In the course of our studies, we came across the already known deeper Tchakaloff results, in particular~\cite[Thm.~5.1]{BerSas12}, which was unknown to us and others, and  the apparent unawareness is the reason why 
we begin with a concise recapitulation. 
While originally formulated in a purely real-valued setting, we extend it in Theorem~\ref{thm:tchak}
to the case of complex-valued functions and provide, in Section~\ref{sec:2}, a short,
well-accessible proof based on elementary convex geometry. 
In addition, we consider the discretization of real and complex measures in Theorems~\ref{thm:quad1} and~\ref{thm:quad2} to work out the common 
structure
behind these different kinds of discretizations.

The central part of this article is Section~3, where we deal with a related, but more general, question, 
namely the search for conditions for positive linear functionals allowing a positive discretization.
Here, there exist counterexamples, see e.g.~Examples~\ref{ex:pos_discr} and~\ref{ex:pos_discr_2}, showing that for the class of positive functionals 
a universal Tchakaloff-type
result like Theorem~\ref{thm:tchak} 
cannot exist.
Until now a precise, universal
characterization of positive discretizability is lacking. 
Our main results of Section~3, Theorems~\ref{thm:pos_functionals} and~\ref{thm:pos_functionals_2}, fill this gap.

The consequences of Theorems~\ref{thm:tchak}, \ref{thm:pos_functionals}, and~\ref{thm:pos_functionals_2}  for several selected 
mathematical subjects are investigated in 
\mbox{Sections~4-6}. 
In Section~4, we provide a connection between the convergence rate of positive quadrature rules on certain compact function classes $\cF\subset L_1(\Omega,\mu;\R)$  and the Kolmogorov widths $d_n(\cF)_{B_\mu(\Omega)}$ and $d_n(\cF)_{L_2(\mu)}$, see~\eqref{eqdef:Kolmo_infty} and~\eqref{eqdef:Kolmo_two} for definitions.  
In Sections~5 and~6, we deal with integral norm discretization
and frame subsampling. The question of how to constructively obtain exact discretization points and weights is the topic of the final Section~7. Here a connection to $D$-optimal design is established, analogous to~\cite[Thm.~3.1]{BKPSU24}, which might be helpful 
for the future development of feasible algorithms.

%%%%%%%%%%%%%%%%%%%%%%%%%%%%%%%%%%%%%%%%%%%%%%%%%%%%%%%

Let us proceed with a brief overview
on the specific contents of this paper.

\paragraph{Exact quadrature.} 
Our main theorem, Theorem~\ref{thm:tchak}, is an extension of Bisgaard's theorem~\cite[Thm.~5.1]{BerSas12} to 
complex-valued functions. 
For its formulation, we need to introduce some notation. 
Given a subset $M$ of a Hilbert space $\cH$, a positive measure $\mu$ on some domain $\Omega$, and $1\le p<\infty$,  let
\begin{align*}%\label{eqdef:Lp_space_equiv}
\cL_p(\Omega,\mu;M) := \Big\{ f:\Omega\to M \text{ $\mu$-measurable} ~:~ \|f\|_{L_p(\mu)} := \Big( \int\limits_\Omega \|f(x)\|_{\cH}^p \dint\mu(x) \Big)^{1/p}<\infty \Big\} \,,
\end{align*}
and $\sim_{\mu}$ be the equivalence relation $f\sim_{\mu} g :\Leftrightarrow \mu( \{ x\in\Omega : f(x)\neq g(x) \} )=0$ on $\mu$-measurable functions. The $p$-Lebesgue space $L_p(\Omega,\mu;M)$ is then defined as, see e.g.~\cite[1.13]{Alt_12},
\begin{align}\label{eqdef:Lp_space}
L_p(\Omega,\mu;M) := \cL_p(\Omega,\mu;M) / \sim_{\mu} \,.
\end{align}
Note that this definition includes the case $M=\cH$ and, in particular, the standard 
case $M=\K$. Further note that, 
if $\cH$ is finite-dimensional, membership in $L_p(\Omega,\mu;\cH)$ can also be described differently. A function $f:\Omega\to\cH$ is in $L_p(\Omega,\mu;\cH)$ if and only if $x\mapsto\langle f(x),\bvarphi\rangle_{\cH}$ is in $L_p(\Omega,\mu;\K)$ for every $\bvarphi\in\cH$. 

We can now state Theorem~\ref{thm:tchak} providing Tchakaloff discretization of $\mu$ in any finite-dimensional subspace of $L_1(\Omega,\mu;\K)$, where both $\K=\R$ and $\K=\C$ are allowed.

\begin{theorem}\label{thm:tchak}
Let $(\Omega,\Sigma,\mu)$ be a positive measure space, whereby $\Omega$ may be any generic set, and $\K\in\{\R,\C\}$. For each given $n$-dimensional subspace $V_n$ of $L_1(\Omega,\mu;\K)$ with effective real dimension $N$, see Definition~\ref{def:eff_dim}, there exist 
points $x_1,\ldots,x_N\in\Omega$ and associated non-negative weights $\mu_1,\ldots,\mu_N\ge 0$ such that
\begin{align*}
\int_{\Omega} f(x) \dint\mu(x) = \sum_{j=1}^{N} \mu_j f(x_j) \quad\text{ for all $f\in V_n$}\,.
\end{align*}
\end{theorem}

Since elements of $L_p(\Omega,\mu;\K)$ are often not proper functions but merely $\mu$-equivalence classes of functions, sampling may not be a feasible operation in $L_p(\Omega,\mu;\K)$ making quadrature rules unconceivable. The following remark comments on this situation.

\begin{remark}%\label{rem:samp_funcvals}
The space $V_n$ in Theorem~\ref{thm:tchak} is assumed to consist of proper functions from $\Omega$ to $\K$, whose function values $f(x)$ are unambiguously defined for each $x\in\Omega$. If $V_n$ is a-priori given as a space of $\mu$-equivalence classes, one needs to choose fixed representatives $\varphi_1,\ldots,\varphi_n$ of a basis of $V_n$ and then consider $V_n$ as the proper function space $\lin_{\K}\{ \varphi_j : j=1,\ldots,n \}$, i.e., as the $\K$-linear hull of the representing basis functions, see Definition~\ref{def:lin_hull}.
\end{remark}

\begin{definition}\label{def:lin_hull}
The \emph{$\K$-linear hull} of a subset $M$ of a $\K$-linear space, where $\K\in\{\R,\C\}$, is defined by
\begin{align*}
\lin_{\K} M := \Big\{ \sum_{j=1}^{N} \lambda_j \bv_j : \bv_j\in M,\, \lambda_j\in\K,\, N\in\N \Big\} \,. 
\end{align*}
It is also common to refer to this hull as \emph{$\K$-span}.
\end{definition}

To make the picture 
complete, we additionally recall in Section~\ref{sec:2} analogous integral discretizations on real and complex measure spaces. These are given in Theorems~\ref{thm:quad1}
and~\ref{thm:quad2} and are standard knowledge, yet 
they nicely illustrate the structural similarities to Theorem~\ref{thm:tchak}. 

\paragraph{Discretization of positive linear functionals.}

The major contribution of this article is a precise characterization
of positively discretizable functionals $L:V_n\to\K$, where $V_n$ is an arbitrary $n$-dimensional linear function space on a generic domain $\Omega$.
These characterizations are given in Theorems~\ref{thm:pos_functionals} and~\ref{thm:pos_functionals_2} of Section~3. 
In Theorem~\ref{thm:pos_functionals}, we prove for real-valued $V_n$ and $\R$-linear $L:V_n\to\R$ that positive discretizability
is equivalent to $L$ being \emph{strictly $S$-positive} for some $S\subseteq\Omega$. 
Strict $S$-positivity of $L$ hereby means $S$-positivity, see Definition~\ref{def:D-positive}, and additionally
\begin{align}\label{def:strict_D-positive}
 f|_S\ge0 \wedge f|_S\neq 0 \enspace\Rightarrow\enspace Lf > 0 \,,
\end{align}
where $f|_S$ denotes the restriction of $f\in V_n$ to $S$. 
An extension of this result to complex-valued $V_n$ and $\C$-linear functionals is provided
by Theorem~\ref{thm:pos_functionals_2}.
As direct corollaries, we obtain from Theorem~\ref{thm:pos_functionals} the real version of the general Tchakaloff theorem, i.e., \cite[Thm.~5.1]{BerSas12} due to Bisgaard, and from Theorem~\ref{thm:pos_functionals_2} its complex-valued extension, Theorem~\ref{thm:tchak}.

\paragraph{Tchakaloff quadrature and Kolmogorov widths.}

A nice direct application of Theorem~\ref{thm:tchak} is presented in Section~\ref{sec:4}.
Here we deduce upper bounds for the rate of convergence of positive quadrature formulas to the exact integral for functions $f\in\cF$ from certain classes
$\cF\subseteq L_1(\Omega,\mu;\R)$. 
As a suitable metric for the approximability of integrals we use the following quadrature widths, cf.~\cite[8.1.2]{DuTeUl18}, 
\begin{align}\label{eqdef:kappa_numbers}
\kappa^{+}_n(\cF;\mu) :=  \inf\limits_{\substack{x_1,\ldots,x_n\in\Omega \\ \mu_1,\ldots,\mu_n\ge 0 }} \:
\sup\limits_{f \in \cF} \:
\Big| \int\limits_{\Omega} f(x)\dint\mu(x) - \sum\limits_{j=1}^{n} \mu_j f(x_j) \Big| \,,
\end{align}
subsequently referred to as \emph{Tchakaloff quadrature widths}, where $n\in\N$ and the infimum is taken over all quadrature rules with $n$ nodes and non-negative weights.

In Theorem~\ref{prop:Kolmogorov}, we prove a relation which complements a classical result~\cite[Prop.~2]{Nov86} by Novak, namely
\begin{align*}
\kappa^{+}_{n+1}(\cF;\mu)  \le 2\mu(\Omega) d_{n}(\cF)_{B_\mu(\Omega)} \quad\text{ for all $n\in\N_{0}$}\,, 
\end{align*}
where $d_n(\cF)_{B_\mu(\Omega)}$ are certain Kolmogorov widths of $\cF$ measured in the uniform norm.  
For the definition of $d_n(\cF)_{B_\mu(\Omega)}$, 
see~\eqref{eqdef:Kolmo_infty}.
In addition, we obtain a result for specific subclasses $\cF\subseteq L_2(\Omega,\mu;\R)$ that appear naturally as the unit balls of separable Hilbert spaces $\cH$ of functions 
on $\Omega$ compactly embedded via
\begin{equation*}%\label{Id_embed}
\Id: \cH \rightarrow L_2(\Omega,\mu;\R) \,.
\end{equation*}
Denoting the 
singular numbers of $\Id$ by $\sigma_1 \geq \sigma_2 \geq \sigma_3 \geq \cdots \ge 0$,
Theorem~\ref{prop:Kolmogorov_2} yields the estimate
\begin{align*}
\kappa^{+}_{n}(\cF;\mu) 
\leq 2\sqrt{\mu(\Omega)} \Big( \sum\limits_{j=n}^{\infty}\sigma_j^2 \Big)^{1/2} \quad\text{ for all $n\in\N$} \,.
\end{align*}
Here the singular numbers correspond to the Kolmogorov widths $d_n(\cF)_{L_2(\mu)}$, see~\eqref{eqdef:Kolmo_two} for the definition.

\paragraph{Integral norm discretization.} 

In Section~\ref{sec:3}, following the exposition in~\cite[Sec.~4]{DaPrTeTi19}, Theorem~\ref{thm:tchak} is used for 
the discretization of $L_p$-norms $\|f\|_{L_p(\mu)}$, see~\eqref{eqdef:Lp_space}, for functions $f\in L_p(\Omega,\mu;\K)$, where $\mu$ is a positive measure and $1\le p<\infty$. 
The aim are so-called $L_p$-Marcinkiewicz-Zygmund (MZ) equalities, often referred to as exact Marcinkiewicz-Zygmund (MZ) inequalities, which have the form
\begin{equation}\label{eqdef:LpMZequalities}
        \int_\Omega |f(x)|^p \dint\mu(x) = \sum_{j=1}^{N} \mu_j |f(x_j)|^p 
\end{equation}
for a finite point set $x_1,\ldots,x_N\in\Omega$ and associated non-negative weights $\mu_1,\ldots,\mu_N$.
A broader survey on MZ (in)equalities is given in~\cite{Tem2018MZ,DaPrTeTi19,KaKoLimTem_22}. Originally introduced by Marcinkiewicz and Zygmund in 1935 for trigonometric spaces, see~\cite{Zyg59},  
they have developed into a wide research area. For some recent contributions see e.g.~\cite{Tem2021HyperbolicCross,LimTem22,DaTe22,BSU23,FiHiJaUl24}, or for $L_2$-equalities with negative weights~\cite{Li21,Li21erratum}. Connections to sampling reconstruction and quadrature are analyzed in~\cite{Gr20,BKPSU24}. For negative results we refer to~\cite{FREEMAN2023126846}.

A fundamental question is, for which subspaces of $L_p(\Omega,\mu;\K)$ equalities of the form~\eqref{eqdef:LpMZequalities} exist.
It is clear, due to the restriction to finite sums on the right-hand side,
that potential subspaces need to be finite-dimensional. 
And it is known, see e.g.~\cite[Prop.\ 3.3]{DaPrTeTi19}, that $L_p$-MZ equalities may fail to exist if $1\le p<\infty$ is not an even integer.
For even $p\in2\N$, on the other hand, the existence has been proven under fairly general assumptions, see e.g.~\cite[Thm~3.2]{DaPrTeTi19}.
Theorem~\ref{thm:tchak} now allows us to formulate a general positive result, valid for any finite-dimensional subspace of $L_p(\Omega,\mu;\K)$ if $p$ is even, see Theorems~\ref{thm:MZ_product_of_functions} and~\ref{even_MZ}.

\paragraph{Frame discretization.}

Section~\ref{sec:5} starts with a result on the orthonormalization of a given set of functions $\varphi_1,\dots,\varphi_n\colon \Omega\to\K$.
This result, Proposition~\ref{prop:discr_orth}, is in the spirit of~\cite[Thm.~3.1]{BKPSU24} and also examines the discretizability 
of orthonormality relations. 
It is the basis for our subsequent analysis of the discretizability of continuous frames, a problem deeply studied in literature, see e.g.~\cite{FS19,FREEMAN2023126846} and the references therein.  
One hereby differentiates between weighted and unweighted discretization. While unweighted discretization refers to a plain subsampling of the frame and is investigated e.g.\ in the seminal article~\cite{FS19} by Freeman and Speegle, we focus in our analysis on weighted discretization, which includes an additional rescaling step that significantly simplifies the problem.

Our main results
are Theorems~\ref{thm:discr_orth} and~\ref{thm:discr_orth_2}, 
giving precise criteria for vector families $\{\bvarphi_x\}_{x\in\Omega}$, when those can be turned into continuous frames for $\K^n$ with prescribed frame operators. This is achieved 
through the choice of appropriate measures $\mu$ on $\Omega$.
In particular, see Corollary~\ref{cor:discr_scalability}, these theorems 
answer the question of frame scalability, a topic handled e.g.\ in~\cite{KUTYNIOK20132225,Cahill_Chen_2013,Kutyniok_Okoudjou_Philipp,Casazza_Chen_2017,CasaCarliTran23}.
Further, see Corollaries~\ref{cor:discr_orth} and~\ref{cor:discr_subframe}, they assure the existence of exact weighted discretizations for any finite-dimensional Hilbert frame $\bvarphi:\Omega\to\cH$, see Definition~\ref{def:mu-frame}. 
More concretely, there always exist a number $N\in\N$, points $x_1,\ldots,x_N\in\Omega$, and scaling weights $s_1,\ldots,s_N>0$ such that the discrete frame
\begin{align*}
\big\{s_j\bvarphi(x_j)\big\}_{j=1,\ldots,N} 
\end{align*}
has the same frame operator as $\bvarphi$.

\paragraph{$D$-optimal design.}

$D$-optimal designs are specific experimental designs, see e.g.~\cite{DeSt97}, where the determinant of the information matrix is maximal. They originate from statistics but
they have also found applications in the construction of good point sets for quadrature and interpolation.   
As an example, let us mention Fekete points, where one maximizes the determinant of Vandermonde-type matrices~\cite{Bo23}. Their close relationship to approximation-theoretic tasks is further highlighted in~\cite{BoPiVi19,BoPiVi20}.
Our investigations in Section~\ref{sec:$D$-opt} show that also the task of finding suitable point-weight sets for the exact discretization of frames can be formulated as a $D$-optimal design problem. In Theorem~\ref{$D$-opt}, our main result and a generalization of~\cite[Thm.~3.1]{BKPSU24}, this task boils down to a determinant maximization problem 
on a set of Gramian matrices.
Although this maximization is usually a difficult, non-convex optimization problem, it  
might be a helpful approach for the development of feasible algorithms.

%%%%%%%%%%%%%%%%%%%%%%%%%%%%%%%%%%%%%%%%%%%%%%%%%%%%%%%%%%

\paragraph{Basic notation.} 

As usual, $\N$, $\Z$,  $\R$, $\C$ denote the natural, integer, real, and complex numbers. The symbol $\K$ always stands for either $\R$ or $\C$. In addition, we introduce $\N_{0}:=\N\cup\{0\}$ and $M_{\ge t}:=\{ x\in M:x\ge t\}$ for $M\subseteq\R$ and $t\in\R$. Analogously, $M_{>t}$ is defined, where the order relation is strict.
$\K^n$ shall denote the $n$-space of column vectors over $\K$ and $\K^{m\times n}$ the corresponding set of $m\times n$-matrices.
We write $\Re(x)$ and $\Im(x)$ for the real and imaginary part of $x\in\K$. Further, we use $|x|$ for the absolute value and $\overline{x}$ for conjugation. 
For an element $\bx\in\K^n$, the vector  $\overline{\bx}$ is obtained from $\bx$ by conjugation in each component. Further, $\bx^\top$ is the associated row vector and $\bx^\ast$ stands for $\overline{\bx}^\top$. Corresponding notation is used for the matrix spaces $\K^{m\times n}$. Norms are denoted by $\|\cdot\|$ and scalar products by $\langle\cdot,\cdot\rangle$, with additional subscripts if necessary.
For a subset $M\subseteq\K^n$, we let $\overline{M}$ denote its closure, $\operatorname{int}M$ its interior, and $\operatorname{\partial}M$ its boundary with respect to the usual topology on $\K^n$.
The collection of all functions $f:\Omega\to\K$ is denoted by $\K^\Omega$. If $\Omega$ is finite, it is without further notification identified with $\K^{|\Omega|}$, where $|\Omega|$ is the cardinality of $\Omega$. The restriction of $f:\Omega\to\K$ to a subset $S\subseteq\Omega$ is written as $f|_S$. Positivity of $f\in\R^\Omega$, meaning $f(x) \ge0$ for every $x\in\Omega$, is written as $f\ge0$. Finally, writing $f>0$ indicates a strictly positive function $f$ with $f(x)>0$ for every $x\in\Omega$. 

%%%%%%%%%%%%%%%%%%%%%%%%%%%%%%%%%%%%%%%%%%%%%%%%%%%%
%--------------------------------------------------%
%%%%%%%%%%%%%%%%%%%%%%%%%%%%%%%%%%%%%%%%%%%%%%%%%%%%

\section{Exact quadrature}
\label{sec:2}

We are first interested in exact quadrature on $n$-dimensional subspaces $V_n$ of  $L_1(\Omega,\mu;\K)$. In Subsection~\ref{ssec:Tchakaloff quadrature} we will prove an extended version of Bisgaard's theorem on Tchakaloff quadrature, the main result of this section, which includes the case of complex-valued functions.

\subsection{Exact quadrature with general weights}

As a starting point, let us recall how a simple dimension argument yields exact quadrature formulas for $V_n\subseteq L_1(\Omega,\mu;\K)$ with $\dim_{\K}V_n=n$ nodes if the weights may be chosen freely from $\K$.
Hereby we allow $\K$-valued measures $\mu$, subsequently called $\K$-measures. 

\begin{definition}
Let $(\Omega,\Sigma)$ be a measurable space and $\K\in\{\R,\C\}$. A function $\mu:\Sigma\to (-\infty,\infty]\times\ii (-\infty,\infty]$ is called a \emph{$\K$-measure} on $(\Omega,\Sigma)$ if it is
$\sigma$-additive, i.e., if it satisfies
\begin{align}\label{eq:sigma-additivity}
\mu\Big( \bigcup_{j=1}^{\infty} S_j \Big) = \sum_{j=1}^{\infty} \mu(S_j)
\end{align}
for any sequence $(S_j)_{j\in\N}$ of disjoint sets belonging to $\Sigma$.
\end{definition}

In case $\K=\C$, the $\K$-measures are usually referred to as \emph{complex measures} in case they are finite, see e.g.~\cite{Rudin_1966,Elstrodt_2018}. We allow them here to take infinite values. In case $\K=\R$, they are called \emph{signed measures}, see e.g.~\cite{Saks_1937,Rudin_1966,Elstrodt_2018}. For distinction, measures in the classical sense with values in $[0,\infty]$ will subsequently be referred to as \emph{positive measures}. A $\K$-measure $\mu$ on $\Omega$ is considered \emph{discrete} if there is an at most countable set $S\subseteq\Omega$, of which each singleton is $\mu$-measurable, such that each $\mu$-measurable subset of $\Omega\setminus S$ has measure $0$. 
We refer to the above mentioned literature for more details.

\begin{remark}
As $\bigcup_{j=1}^{\infty} S_j =  \bigcup_{j=1}^{\infty} S_{\sigma(j)}$ for any permutation $\sigma:\N\to\N$, it follows that $\sum _{j=1}^{\infty}\mu (S_{j})$ in~\eqref{eq:sigma-additivity} converges unconditionally for any $\K$-measure $\mu$. 
\end{remark}

The following fundamental theorem is standard  knowledge
and straight-forward to prove. Hereby we let, for a $\K$-measure $\mu$,
\[
L_1(\Omega,\mu;\K) := L_1(\Omega,|\mu|;\K) \,,
\]
where $|\mu|$ is the total variation of $\mu$.

\begin{theorem}\label{thm:quad1}
Let $(\Omega,\Sigma,\mu)$ be a $\K$-measure space, whereby $\Omega$ may be any generic set, and $\K\in\{\R,\C\}$. For each given $n$-dimensional subspace $V_n\subseteq L_1(\Omega,\mu;\K)$, there exists a set of $n$ points $x_1,\ldots,x_n\in\Omega$ and associated numbers $\mu_1,\ldots,\mu_n\in\K$ such that
\begin{align*}
\int_{\Omega} f(x) \dint\mu(x) = \sum_{j=1}^{n} \mu_j f(x_j) \quad\text{ for all $f\in V_n$}\,.
\end{align*}
\end{theorem}
\begin{proof}
First choose a basis of $n$ linearly independent  functions in $V_n$ and then corresponding representatives $\varphi_1,\ldots,\varphi_n:\Omega\to\K$, which are defined point-wise on $\Omega$ and are also linearly independent. Then apply Lemma~\ref{lem:column_rank}
to obtain nodes $x_1,\ldots,x_n\in\Omega$ such that the system matrix in~\eqref{eq:equation} is regular. There are then unique weights $\mu_1,\ldots,\mu_n\in\K$ which solve the equation 
\begin{equation}\label{eq:equation}
        \begin{pmatrix}
            \varphi_1(x_1) & \dots & \varphi_1(x_{n}) \\
            \vdots & \ddots & \vdots \\
            \varphi_n(x_1) & \dots & \varphi_n(x_{n}) 
        \end{pmatrix}
        \cdot \begin{pmatrix}
            \mu_1 \\
            \vdots \\
            \mu_n 
        \end{pmatrix}   =
        \begin{pmatrix}
            \int_\Omega \varphi_1 \dint\mu \\
            \vdots \\
            \int_\Omega \varphi_n \dint\mu 
        \end{pmatrix}  \,.
    \end{equation}
The quadrature formula is now a consequence of the linearity of the integration operation.
\end{proof}

The lemma below is the key ingredient in the proof of Theorem~\ref{thm:quad1}, where it is applied with $N=n$. Its statement is a generalization of the coincidence of row-rank and column-rank of rectangular matrices. Here the rows are extended to functions on some arbitrary index set $\Omega$.  

\begin{lemma}\label{lem:column_rank}
Let $\varphi_1,\ldots,\varphi_n:\Omega\to\K$ be functions and let
\begin{align*}
N := \dim_{\K} \lin_{\K} \{ \varphi_j : j=1,\ldots,n \} \le n\,.
\end{align*}
There exist $x_1,\ldots,x_N\in \Omega$ such that the matrix
\begin{align*}%\label{reduced_matrix}
\begin{pmatrix}
            \varphi_1(x_1) & \dots & \varphi_1(x_N) \\
            \vdots & \ddots & \vdots \\
            \varphi_n(x_1) & \dots & \varphi_n(x_N) 
        \end{pmatrix}  
\end{align*}
has full column-rank, i.e.\ linearly independent column vectors. Furthermore, $N$ is the maximal number of linearly independent vectors in the set $\{ \bvarphi(x):x\in\Omega \}$, where $\bvarphi(x):=(\varphi_1(x),\ldots,\varphi_n(x))^\top\in\K^{n}$. In other words,
\begin{align}\label{eq:dim_equal}
\dim_{\K} \lin_{\K} \{ \varphi_j : j=1,\ldots,n \}  =\dim_{\K} \lin_{\K} \{ \bvarphi(x):x\in\Omega \} \,.
\end{align}
\end{lemma}
\begin{proof}
Assume that we have $k$ nodes $x_1,\ldots,x_k$
such that the vectors ${\bvarphi}(x_1),\ldots,{\bvarphi}(x_k)$ are linearly independent, whereby $k$ shall be the maximal possible number of linearly independent columns. This ensures that ${\bvarphi}(x) \in \lin_{\K}\{ {\bvarphi}(x_j) : j=1,\ldots,k \}$ for every $x\in\Omega$. As a consequence, there exist $\lambda_1,\ldots,\lambda_k\in\K$ for every $x\in\Omega$ with 
\begin{align}\label{aux_label}
{\bvarphi}(x)= \sum_{j=1}^{k} \lambda_j  \bvarphi(x_j) \,. 
\end{align}
The column-rank of the matrix
\begin{align*}
\begin{pmatrix}
            \varphi_1(x_1) & \dots & \varphi_1(x_k) \\
            \vdots & \ddots & \vdots \\
            \varphi_n(x_1) & \dots & \varphi_n(x_k) 
        \end{pmatrix} 
\end{align*}
is $k$ and the row-rank $r$ satisfies $r\le N$, which follows right from the definition of $N$.
Since the row-rank and the column-rank of a matrix always coincide, we thus get $k\le N$.
To see that $k<N$ is not possible, define
\begin{align*}
D := \Big\{ (c_1,\ldots,c_n)^\top \in \K^n: \sum_{j=1}^{n} c_j \varphi_j = 0  \Big\} \,.
\end{align*}
Then, 
%$\dim_{\K} D = n-N$ and, 
if $k<N$, there exists a coefficient vector $(c_1,\ldots,c_n)^\top\in\K^{n}\setminus D$ such that 
\begin{align*}
\sum_{i=1}^{n} c_i\varphi_i(x_j) = 0 \quad\text{ for all $j=1,\ldots,k$}\,.
\end{align*}
But, in view of~\eqref{aux_label}, we can deduce for any $x\in\Omega$  
\begin{align*}
\sum_{i=1}^{n} c_i\varphi_i(x) = 
\sum_{i=1}^{n} c_i \sum_{j=1}^{k} \lambda_j  \varphi_i(x_j) = \sum_{j=1}^{k} \lambda_j \sum_{i=1}^{n} c_i\varphi_i(x_j) = 0 \,,
\end{align*}
which implies $(c_1,\ldots,c_n)^\top\in D$, a contradiction. The proof is finished.
\end{proof}

Let us also record the subsequent corollary of
Lemma~\ref{lem:column_rank}.

\begin{corollary}\label{cor:column_rank}
Let $\varphi_1,\ldots,\varphi_n:\Omega\to\K$ be functions and $\bvarphi(x):=(\varphi_1(x),\ldots,\varphi_n(x))^\top\in\K^n$. Further, let $x_1,\ldots,x_N$ be the points from Lemma~\ref{lem:column_rank}. 
Then
\[
\lin_{\K}\{ \bvarphi(x):x\in\Omega \} = \lin_{\K}\{ \bvarphi(x_j):j=1,\ldots,N\}  \,.
\]
\end{corollary}
\begin{proof}
The $N$ vectors $\bvarphi(x_1),\ldots,\bvarphi(x_N)$
are linearly independent in $\K^n$, according to Lemma~\ref{lem:column_rank}, and $N=\dim_{\K} \lin_{\K} \{ \bvarphi(x) : x\in\Omega \}$ is the maximal possible number of linearly independent vectors in $\{ \bvarphi(x) : x\in\Omega \}$. Hence,
$\bvarphi(x)\in\lin_{\K}\{ \bvarphi(x_j):j=1,\ldots,N\}$ for every $x\in\Omega$, implying $\lin_{\K}\{ \bvarphi(x):x\in\Omega \} \subseteq \lin_{\K}\{ \bvarphi(x_j):j=1,\ldots,N\}$. The opposite inclusion $\lin_{\K}\{ \bvarphi(x):x\in\Omega \} \supseteq \lin_{\K}\{ \bvarphi(x_j):j=1,\ldots,N\}$ is obvious.  
\end{proof}

We next present a  
slightly modified approach to the proof of Theorem~\ref{thm:quad1}, so that its principles can be generalized to the case of Tchakaloff quadrature. 
To this end, we introduce the following notion.

\begin{definition}\label{def:int_lin_hull}
Let $M\subseteq\K^n$, $n\in\N$, and $\K\in\{\R,\C\}$.
The \emph{integral $\K$-linear hull} of $M$ is defined by
\begin{align*}
\intlin_{\K} M := \bigg\{ \int_{\Omega} \bv(x)\dint\mu(x) : (\Omega,\Sigma,\mu) \text{ $\K$-measure space},\, \bv\in L_{1}(\Omega,\mu;M)  \bigg\} \,.
\end{align*}
\end{definition}

Since one can choose $\Omega=\N$ and a $\K$-valued discrete measure $\mu$ on $\N$, one directly sees the inclusion
\begin{align}\label{lin_hull_inclusions}
\lin_{\K} M \subseteq \intlin_{\K} M  \,.
\end{align}
But we even have the following result. 

\begin{lemma}\label{lem:intlin_eq_lin}
It holds for any set $M\subseteq\K^n$, where $n\in\N$ and $\K\in\{\R,\C\}$, 
\begin{align*}
\intlin_{\K} M = \lin_{\K} M  \,.
\end{align*}
\end{lemma}
\begin{proof}
Due to~\eqref{lin_hull_inclusions}, it only remains to prove the opposite inclusion. We can describe $\lin_{\K} M$ dually by linear functionals. Let us put $Z:=M^\perp$. Then
\begin{align*}
  \lin_{\K} M = Z^\perp.
\end{align*}
Further, for $\bv_\mu:=\int_{\Omega} \bv(x)\dint\mu(x) \in\intlin_{\K} M$ and every $\bn\in Z$ it holds
\begin{align*}
\langle \bn, \bv_\mu \rangle =  \int_{\Omega} \langle \bn,\bv(x) \rangle \dint\mu(x) = 0 \,.
\end{align*}
As a consequence, $\bv_\mu\in \lin_{\K} M$.
\end{proof}

Lemma~\ref{lem:intlin_eq_lin} allows to formulate a generalization of Corollary~\ref{cor:column_rank}. 

\begin{proposition}[Dimension-based subsampling]
\label{ContSub} %%%
    Let $\varphi_1,\ldots,\varphi_n:\Omega\to\K$ be functions and $\bvarphi(x) := (\varphi_1(x),\ldots,\varphi_n(x))^\top\in\K^{n}$.
    Then there exist 
    \[
    N := \dim_{\K}\lin_{\K}\{\varphi_j:j=1,\ldots,n\}
    \]
    points $x_1,\ldots,x_N\in\Omega$ so that
    \begin{equation*}
        \intlin_{\K} \{ \bvarphi(x) : x\in\Omega \} = \lin_{\K}\{ \bvarphi(x_j): j=1,\ldots,N \} \,.
    \end{equation*}
\end{proposition} 
\begin{proof}
By Lemma~\ref{lem:intlin_eq_lin} we have
\[
\intlin_{\K} \Big\{ \bvarphi(x)  : x\in\Omega \Big\} = \lin_{\K} \Big\{ \bvarphi(x)  : x\in\Omega \Big\} \,.
\]
Further, due to Corollary~\ref{cor:column_rank}, there are points $x_1,\ldots,x_N\in\Omega$ such that
\begin{equation*}
           \lin_{\K} \Big\{ \bvarphi(x)  : x\in\Omega \Big\} = \lin_{\K}\Big\{ \bvarphi(x_j):j=1,\ldots,N\Big\} \,. \qedhere
\end{equation*}
\end{proof}

This last proposition is a powerful tool that
allows to reformulate
the proof of Theorem~\ref{thm:quad1} in such a way
that it essentially becomes a sequence of $4$ simple 
steps. 

\begin{tcolorbox}
[ colback=white,
  colframe=black,
  arc=0mm,
  boxrule=1pt,
  fonttitle=\bfseries
]
\begin{proof}[\textbf{\textup{Modified proof of Theorem~\ref{thm:quad1} (based on Proposition~\ref{ContSub})}}]

~\\[0.8ex]
%\vspace*{0.1ex}
\emph{Step~1.} Choose a $\K$-linear basis $\varphi_1,\ldots,\varphi_n:\Omega\to\K$ of $V_n$ and build $\bvarphi(x) := (\varphi_1(x),\ldots,\varphi_n(x))^\top$.

\emph{Step~2.} Observe that $\bvarphi_{\mu} :=
            \int_\Omega \bvarphi(x) \dint\mu(x)   \in \intlin_{\K}\{ \bvarphi(x) : x\in\Omega \}$.
    
\emph{Step~3.} Apply Proposition~\ref{ContSub}, which yields nodes $x_1,\ldots,x_{N}\in\Omega$ with \[
\bvarphi_{\mu}\in \lin_{\K}\{ \bvarphi(x_j): j=1,\ldots,N \} \,.
\]

\emph{Step~4.}   The quadrature formula follows by linearity.
\end{proof}
\end{tcolorbox}

We will see that Tchakaloff's theorem can be proved
based on the same proof idea.
But before we turn to Tchakaloff, let us consider exact quadrature with real weights.

%--------------------------------------------------------

\subsection{Exact quadrature with real weights}

If we only 
allow real weights in the quadrature formula, some adaption is necessary in Theorem~\ref{thm:quad1} in case $\K=\C$,  
wherefore we define the $\R$-linear projection $\Re$ which maps functions $f:\Omega\to\K$ 
to functions $\Re(f):\Omega\to\R$ 
via
\begin{align}\label{eqdef:Re-Op}
\Re(f) := \Re\circ f \,.
\end{align}
Further, for the $n$-dimensional 
subspace $V_n$ of $L_1(\Omega,\mu;\K)$ let us consider the associated $\R$-vector space
\[
\Re(V_n) := \Big\{ \Re(f)  : f\in V_n \Big\} \,,
\]
and let us make the following definition.

\begin{definition}\label{def:eff_dim}
 Let $\K\in\{\R,\C\}$. 
 The \emph{effective real dimension} of $V_n$ is given by
 \begin{align*}
 \effdim_{\R} V_n := \dim_{\R} \operatorname{Re}(V_n) \,.
 \end{align*}
\end{definition}

In case $\K=\R$, it clearly holds $\effdim_{\R}V_n=\dim_{\R}V_n$. In case $\K=\C$, the effective dimension $\effdim_{\R}V_n$ is in the range
\begin{align*}
\dim_{\C} V_n \le \effdim_{\R}V_n \le 2 \dim_{\C} V_n \,.
\end{align*}

We have the following theorem. 

\begin{theorem}\label{thm:quad2}
Let $(\Omega,\Sigma,\mu)$ be a measure space with real measure $\mu$, whereby $\Omega$ may be any generic set, and let $\K\in\{\R,\C\}$. For each $n$-dimensional subspace $V_n\subseteq L_1(\Omega,\mu;\K)$ there exists a set of $N$ points $x_1,\ldots,x_N\in\Omega$, $N$ being the effective real dimension of $V_n$, see Definition~\ref{def:eff_dim}, and associated real numbers $\mu_1,\ldots,\mu_N\in\R$ such that
\begin{align*}
\int_{\Omega} f(x) \dint\mu(x) = \sum_{j=1}^{N} \mu_j f(x_j) \quad\text{ for all $f\in V_n$}\,.
\end{align*}
\end{theorem}
\begin{proof}
If $\K=\R$, Theorem~\ref{thm:quad1} already provides the statement since $\effdim_{\R}V_n=\dim_{\R}V_n$.

To prove the case $\K=\C$,
let us define the $\R$-linear projection $\Im$ given by
\begin{align}\label{eqdef:Im-Op}
\Im(f) :=  \Im\circ f \,.
\end{align}
Then, like $\Re(V_n)$, also $\Im(V_n) := \{ \Im(f)  : f\in V_n \}$ is an $\R$-vector space which happens
to be equal to $\Re(V_n)$, i.e., it holds $\Re(V_n)=\Im(V_n)$.
Hence, an exact quadrature rule on $\Re(V_n)$ 
is automatically
exact on $\Im(V_n)$. The assertion is thus again a consequence of Theorem~\ref{thm:quad1}, applied to the real function space $\Re(V_n)\subseteq L_1(\Omega,\mu;\R)$.
\end{proof}

\begin{remark}
In case $\K=\R$, one can find $N=n$ nodes in Theorem~\ref{thm:quad2}. In case $\K=\C$, the number $N$ is in the range $n\le N \le 2n$.
\end{remark}

%---------------------------------------------------

\subsection{Tchakaloff quadrature}
\label{ssec:Tchakaloff quadrature}

When one strives for quadrature with non-negative weights, as provided by the original theorem of Tchakaloff from 1957 and subsequent generalizations,
things are more complicated. It took quite a while to arrive at its most general form~\cite[Thm.~5.1]{BerSas12} proved by Bisgaard in 2012. Essentially, this result is recited in Theorem~\ref{thm:tchak}, the latter being an extension to complex-valued functions. A further extension, handling the question of positive discretizability of linear functionals, is proved in the next section in Theorems~\ref{thm:pos_functionals} and~\ref{thm:pos_functionals_2}.
In this subsection, we provide a short concise proof of Theorem~\ref{thm:tchak}, fully self-contained and with rather elementary arguments.
As we will see, the proof idea can be extracted from the proofs of the preceding Theorems~\ref{thm:quad1} and~\ref{thm:quad2}.

The main challenge is the adaption of the steps in the second proof of Theorem~\ref{thm:quad1} to the convex geometric situation, i.e., to replace:
\begin{enumerate}[label=\normalfont(\roman*)]
\item (Integral) linear hull (Def.~\ref{def:lin_hull} \& Def.~\ref{def:int_lin_hull}) $\rightsquigarrow$ (Integral) conic hull (Def.~\ref{def:conic_hull} \& Def.~\ref{def:int_conic_hull}).
\item Dual description of subspaces (in the proof of Lem.~\ref{lem:intlin_eq_lin}) $\rightsquigarrow$ Supporting hyperplane theorem (for the proof of Lem.~\ref{lem:intcone_eq_cone}, see e.g.~\cite[pp.~50-51]{BoydVandenberghe2004}). 
\item 
Dimension-based subsampling (Prop.~\ref{ContSub}) $\rightsquigarrow$ Carathéodory-Tchakaloff subsampling (Prop.~\ref{prop:ContCarSub}).
\end{enumerate}

We begin with the following definition.

\begin{definition}\label{def:conic_hull}
Let $M\subseteq\K^n$ be a set, where $n\in\N$ and $\K\in\{\R,\C\}$.
\begin{enumerate}[label=\normalfont(\roman*)]
\item The \emph{convex hull} of $M$ is defined by
\begin{align*}
\conv M := \Big\{ \sum_{j=1}^{N} \lambda_j \bv_j : \bv_j\in M,\, \lambda_j\ge0,\,\sum_{j=1}^{N} \lambda_j=1,\, N\in\N \Big\} \,. 
\end{align*}
\item The \emph{conic hull} of $M$ is defined by
\begin{align*}
\cone M := \Big\{ \sum_{j=1}^{N} \lambda_j \bv_j : \bv_j\in M,\, \lambda_j\ge0,\, N\in\N \Big\} \,.
\end{align*}
\end{enumerate}
\end{definition}

We further introduce the corresponding integral versions.

\begin{definition}\label{def:int_conic_hull}
Let $M\subseteq\K^n$ be a set, where $n\in\N$ and $\K\in\{\R,\C\}$.
\begin{enumerate}[label=\normalfont(\roman*)]
\item The \emph{integral convex hull} of $M$ is defined by
\begin{align*}
\intconv M := \bigg\{ \int_{\Omega} \bv(x)\dint\mu(x) : (\Omega,\Sigma,\mu) \text{ probability measure space},\, \bv\in L_{1}(\Omega,\mu;M)  \bigg\} \,.
\end{align*}
\item The \emph{integral conic hull} of $M$ is defined by
\begin{align*}
\intcone M := \bigg\{ \int_{\Omega} \bv(x)\dint\mu(x) : (\Omega,\Sigma,\mu) \text{ positive measure space},\, \bv\in L_{1}(\Omega,\mu;M)  \bigg\} \,.
\end{align*}
\end{enumerate}
\end{definition}

Since one can choose $\Omega=\N$ and the counting measure $\mu$ on $\N$, one directly sees the inclusions
\begin{align}\label{hull_inclusions}
\conv M \subseteq \intconv M \quad,\quad \cone M \subseteq \intcone M \,.
\end{align}

We now aim for an analogous result as Lemma~\ref{lem:intlin_eq_lin}.
At first sight, it may come as a surprising fact that indeed the integral notions of the hull are equal to the corresponding classical finite notions. For the proof of this fact, we need to recall some facts from convex geometry, see e.g.~\cite[pp.~50-51]{BoydVandenberghe2004}.

\begin{definition}\label{def:support_hyperplane}
A \emph{supporting hyperplane} $\cE$ of a set $M\subseteq\R^n$ is a hyperplane that has both of the following two properties:
\begin{enumerate}[label=\normalfont(\roman*)]
\item
    $M$ is entirely contained in one of the two closed half-spaces bounded by $\cE$.
\item
    $M$ has at least one boundary-point on $\cE$.
\end{enumerate}
\end{definition}

\begin{theorem}[Supporting hyperplane theorem]\label{thm:support_hyperplane}
 If $M$ is a convex set in $\R^n$, $m_0$ is a point on the boundary of $M$, then there exists a supporting hyperplane of $M$ containing $m_0$.
\end{theorem}

With this theorem, we can prove the desired statement.

\begin{lemma}\label{lem:intcone_eq_cone}
It holds for any set $M\subseteq\K^n$, where $n\in\N$ and $\K\in\{\R,\C\}$, 
\begin{align*}
(i)\quad\intconv M = \conv M \,,\quad
(ii)\quad \intcone M = \cone M \,.
\end{align*}
\end{lemma}
\begin{proof}
Due to~\eqref{hull_inclusions}, it only remains to prove the opposite inclusions. Let us first assume $\K=\R$. 
We can then prove (i) by induction on the dimension $n$.

$n=1$:\,
Since for $n=1$, in $\R$, the convex sets are precisely the intervals in $\R$, the assertion is clearly true for integrals over probability measures.

$n\rightsquigarrow n+1$:\, 
Assume $\bv_{\mu}=\int_{\Omega} \bv(x) \dint\mu(x)$ in $\R^{n+1}$ with $\bv\in L_1(\Omega,\mu;M)$, where $\mu$ is a probability measure. Our aim is to show $\bv_{\mu}\in C:=\conv M$. Clearly $\bv_{\mu}\in\overline{C}$, where $\overline{C}$ is the closure of $C$ in $\R^{n+1}$, which is also a convex set. 
We now distinguish the cases $\bv_{\mu}\in\operatorname{int} C$ and $\bv_{\mu}\in\operatorname{\partial} C=\overline{C}\backslash \operatorname{int} C$. 

In case $\bv_{\mu}\in\operatorname{int} C$, there is nothing more to show since $\bv_{\mu}\in C$ is a direct consequence. In case $\bv_{\mu}\in\partial C$, there exists a supporting hyperplane $\cE\subset\R^{n+1}$, see Def.~\ref{def:support_hyperplane} and Thm.~\ref{thm:support_hyperplane},  with $\bv_{\mu}\in\cE$
and $\cE\cap \operatorname{int} C=\emptyset$. Let $\bn\in\R^{n+1}$ be a normal vector pointing into $C$. Then $ \langle \bv(x), \bn\rangle \ge 0$ for all $x\in\Omega$ and
\[
0 = \langle \bv_{\mu}, \bn \rangle = \int_{\Omega}  \langle \bv(x), \bn\rangle \dint\mu(x)  \quad\Rightarrow\quad \langle \bv(x), \bn\rangle = 0 \quad\text{$\mu$-almost everywhere}\,.
\]
Hence, there is a $\mu$-null set $\cN_{\bv}\subseteq\Omega$ such that $\bv(x)\in\cE$ for all $x\in\Omega\backslash\cN_{\bv}$.
We then define a new function 
\begin{align*}
    \tilde{\bv}(x):=\begin{cases} \bv(x) \quad&,\, x\in \Omega\backslash\cN_{\bv} \,, \\ {\bm 0} &,\, x\in\cN_{\bv} \,, \end{cases}
\end{align*}
which is measurable with respect to the completion $\tilde{\mu}$ of $\mu$. 
We then have
$\bv_{\mu}=\int_{\Omega} \tilde{\bv}(x) \dint\tilde{\mu}(x)$, where, in contrast to $\bv$, the range of $\tilde{\bv}$ is fully contained in $\cE$.
Via a linear projection $T$ of $\cE$ onto $\R^n$, we further get
\[
T\bv_{\mu} = \int_{\Omega}  T\tilde{\bv}(x) \dint\tilde{\mu}(x)
\]
with $T\tilde{\bv}(x)\in\R^{n}$.
Therefore, by the induction hypothesis, 
\[
T\bv_{\mu} \in \intconv \{ T\tilde{\bv}(x) : x\in\Omega\}=\conv\{ T\tilde{\bv}(x) : x\in\Omega\} \,,
\]
which implies that there is a number $N\in\N$ with
\[
T\bv_{\mu} = \sum_{j=1}^{N} \lambda_j T\tilde{\bv}(x_j) 
\]
for certain points $x_1,\ldots,x_N\in\Omega$ and weights $\lambda_1,\ldots\lambda_N\ge0$ summing up to $1$. Since points $x_j\in\cN_{\bv}$ do not contribute to this sum, we can additionally assume $x_j\in\Omega\backslash\cN_{\bv}$, whence $\bv(x_j)=\tilde{\bv}(x_j)\in M\cap\cE$. Recall that also $\bv_{\mu}\in\cE$,
wherefore $\bv_{\mu}=  \sum_{j=1}^{N} \lambda_j \bv(x_j)\in \conv M$ follows, proving (i) for $\K=\R$.

We next deduce (ii) from (i) in case $\K=\R$. 
For $\bv_{\mu}\in\intcone M$ there exists a $\sigma$-algebra $\Sigma$ and a positive
measure $\mu:\Sigma\to\R_{\ge0}$ and a function
$\bv\in L_1(\Omega,\mu;M)$ such that
\begin{align*}
\bv_{\mu}=\int_{\Omega}  \bv(x) \dint\mu(x) \,.
\end{align*}
The support $\Omega_{\bv}:=\supp\bv$ is $\sigma$-finite. 
Hence, there exists a
function $\omega\in L_1(\Omega_{\bv},\mu;\R_{>0})$ of integral $1$, which allows to rewrite $\bv_{\mu}\in\intcone M$ in the modified form
\[
\bv_{\mu}=\int_{\Omega}  \bv(x) \dint\mu(x)=\int_{\Omega_{\bv}}  \frac{\bv(x)}{\omega(x)} \dint\nu(x) \,,
\]
using the probability measure $\mathrm{d}\nu:=\omega\cdot\mathrm{d}\mu$.
Hence, in view of (i), 
\[
\bv_{\mu}\in \intconv\Big\{ \frac{\bv(x)}{\omega(x)} : x\in\Omega_{\bv} \Big\} = \conv\Big\{ \frac{\bv(x)}{\omega(x)} : x\in\Omega_{\bv} \Big\} \,.
\]
This proves (ii) since there is then $N\in\N$, $x_1,\ldots,x_N\in\Omega_{\bv}$, and $\lambda_1,\ldots,\lambda_N\ge0$ such that
\[
\bv_{\mu}=\sum_{j=1}^{N}  \frac{\lambda_j}{\omega(x_j)} \bv(x_j) \in \cone\{ \bv(x) : x\in\Omega \}  \,.
\]
For the case $\K=\C$ in (i) and (ii)  we identify $\C^{n}\cong \R^{2n}$.
\end{proof}

Next, recall Carathéodory's hull theorem, which is usually formulated for subsets of $\R^{n}$, see~\cite{Caratheodory11}. We state it for subsets of $\K^n$.

\begin{theorem}[Carathéodory theorem in $\K^{n}$]\label{CarClassic}
Let $M\subseteq \K^{n}$ with 
real dimension $N:=\dim_{\R} \lin_{\R} M$. If $x\in\cone M$,
then $x$ is the non-negative sum of at most 
$N$ points of $M$.
If $x\in\conv M$, then 
$x$ is the convex sum of at most 
$N+1$ points of $M$.
\end{theorem}

The next lemma follows from Theorem~\ref{CarClassic} and corresponds to Corollary~\ref{cor:column_rank}.

\begin{lemma}[%Continuous 
Carathéodory subsampling]\label{CarSub} %%%
    Let $\varphi_1,\ldots,\varphi_n:\Omega\to\K$ be functions and put $\bvarphi(x) = (\varphi_1(x),\ldots,\varphi_n(x))^\top$.
    Let further $N := \dim_{\R}\lin_{\R}\{\bvarphi(x):x\in\Omega\}$.

    \begin{enumerate}[label=\normalfont(\roman*)]
    \item Then
    \begin{equation*}
        \cone\Big\{ \bvarphi(x) : x\in\Omega \Big\} 
        = \bigcup_{x_1,\ldots,x_N\in\Omega} \cone\Big\{ \bvarphi(x_j) : j=1,\ldots,N \Big\}\,.
    \end{equation*}
    \item
        Further
    \begin{equation*}
        \conv\Big\{ \bvarphi(x) : x\in\Omega \Big\} = \bigcup_{x_1,\ldots,x_{N+1}\in\Omega} \conv\Big\{ \bvarphi(x_j) : j=1,\ldots,N+1 \Big\}\,.
    \end{equation*}
    \end{enumerate}
\end{lemma} 

Together with Lemma~\ref{lem:intcone_eq_cone} we get from Lemma~\ref{CarSub} a general continuous version of Carathéodory's theorem. Such continuous versions are in the literature referred to as Carathéodory-Tchakaloff theorems. Algorithmic aspects of the corresponding subsampling procedure are investigated e.g.\ in~\cite{PiSoVi17}.

\begin{proposition}[Carathéodory-Tchakaloff subsampling]\label{prop:ContCarSub} %%%
    Let $\varphi_1,\ldots,\varphi_n:\Omega\to\K$ be functions and put $\bvarphi(x) := (\varphi_1(x),\ldots,\varphi_n(x))^\top$.
    Let further $N := \dim_{\R}\lin_{\R}\{\bvarphi(x):x\in\Omega\}$.

    \begin{enumerate}[label=\normalfont(\roman*)]
    \item For each $\displaystyle{\bvarphi_\mu \in\intcone\{ \bvarphi(x):x\in\Omega \} }$
        there exists a set $\{x_1,\ldots,x_{N} \} \subseteq \Omega$  of nodes such that
        \begin{equation*}
            \bvarphi_\mu \in \cone\{ \bvarphi(x_j): j=1,\ldots,N \}   \,.
        \end{equation*}
    \item For each $\displaystyle{\bvarphi_\mu \in\intconv\{ \bvarphi(x):x\in\Omega \}}$
        there exists a set $\{x_1,\ldots,x_{N+1} \} \subseteq \Omega$  of nodes such that
        \begin{equation*}
            \bvarphi_\mu \in \conv\{ \bvarphi(x_j): j=1,\ldots,N+1 \}   \,.
        \end{equation*}
    \end{enumerate}
\end{proposition} 
\begin{proof}
(i)\:
By Lemma~\ref{lem:intcone_eq_cone} we have
\[
\bvarphi_\mu \in \intcone \Big\{ \bvarphi(x)  : x\in\Omega \Big\} = \cone \Big\{ \bvarphi(x) : x\in\Omega \Big\} \,.
\]
By Lemma~\ref{CarSub}
there are points $x_1,\ldots,x_N\in\Omega$ such that
\begin{equation*}
           \bvarphi_\mu \in\cone\{\bvarphi(x_j):j=1,\ldots,N\} \,.
\end{equation*}

(ii)\: Analogous.
\end{proof}

Since with Proposition~\ref{prop:ContCarSub} we now have a similar tool as Proposition~\ref{ContSub},
we are at last able to prove 
Theorem~\ref{thm:tchak} along the lines of the second proof of Theorem~\ref{thm:quad1}. 
The auxiliary lemma below gives an alternative description of the effective dimension.

\begin{lemma}\label{lem:effdim}
Let $\varphi_1,\ldots,\varphi_n:\Omega\to\K$ be functions and let
$V_n : =\lin_{\K} \{\varphi_1,\ldots,\varphi_n \}$. 
Further, put $\bvarphi(x):=(\varphi_1(x),\ldots,\varphi_n(x))^\top\in\K^{n}$.
Then
\begin{align*}
 \effdim_{\R}V_n 
 %= \dim_{\R} \Re(V_n) 
 = \dim_{\R}\lin_{\R}\{\bvarphi(x):x\in\Omega\} \,.
\end{align*}
\end{lemma}
\begin{proof}
We have $\effdim_{\R}V_n=\dim_{\R} \Re(V_n)$ and
\[
\Re(V_n)=\lin_{\R}\{\Re\circ\varphi_1,\Im\circ\varphi_1,\ldots,\Re\circ\varphi_n,\Im\circ\varphi_n\} \,.
\]
With Lemma~\ref{lem:column_rank}, we can thus argue
\begin{align*}
\effdim_{\R}V_n 
&=\dim_{\R}\lin_{\R}\{  (\Re(\varphi_1(x)),\Im(\varphi_1(x)),\ldots,\Re(\varphi_n(x)),\Im(\varphi_n(x)))^\top \in\R^{2n} : x\in\Omega \}  \\
&=\dim_{\R}\lin_{\R}\{  \bvarphi(x)\in\K^{n} : x\in\Omega \} \,. ~\qedhere
\end{align*}
\end{proof}

\begin{tcolorbox}
[ colback=white,
  colframe=black,
  arc=0mm,
  boxrule=1pt,
  fonttitle=\bfseries
]
\begin{proof}[\textbf{\textup{Proof of Theorem~\ref{thm:tchak}}}]

~\\[0.8ex]
\emph{Step~1.}
     Choose a basis $\varphi_1,\ldots,\varphi_n$
     of $V_n$ consisting of proper functions.
     
\emph{Step~2.}
     Define
     $\bvarphi(x):=(\varphi_1(x),\ldots,\varphi_n(x))^\top$ for $x\in\Omega$.
     Then
     \[
     \bvarphi: \Omega \to\K^{n} \quad\text{ is component-wise $\mu$-integrable}\,.
     \]
     We can hence define
     \[
     \bvarphi_\mu := \int_\Omega \bvarphi(x) \dint\mu(x) \in \intcone\{ \bvarphi(x) : x\in\Omega \} \,.
     \]
\emph{Step~3.}
    Taking Lemma~\ref{lem:effdim} into account, we then use Proposition~\ref{prop:ContCarSub}.
    
\emph{Step~4.}
   The assertion follows by linearity.
\end{proof}
\end{tcolorbox}

\subsection{Additional remarks}

We end this section with a few remarks. The first concerns the case of finite measures.

\begin{remark}\label{rem:finite_measure}
If $\mu$ in Theorems~\ref{thm:tchak},~\ref{thm:quad1},~\ref{thm:quad2} is assumed to be a finite measure, i.e., if $|\mu|(\Omega)<\infty$, then there exist $N+1$ nodes $x_1,\ldots,x_{N+1}\in\Omega$ with associated weights that fulfill the extra condition
\[
\sum_{j=1}^{N+1} \mu_j = \mu(\Omega) \,.
\]
\end{remark}
\begin{proof}
Let $\mathds 1_\Omega$ denote the $1$-function on $\Omega$, which is integrable if $|\mu|(\Omega)<\infty$. We can thus apply the theorems to the expanded space $V_{n+1} := V_n + \lin_{\K}\{ \mathds 1_\Omega \}\subseteq L_1(\Omega,\mu;\K)$ with
\begin{gather*}
    %V_{n+1} := V_n + \lin_{\K}\{ \mathds 1_\Omega \}  \quad\text{with}\quad  
    \effdim_{\R}V_{n+1}\le \effdim_{\R}V_n+1 \quad\text{and}\quad \dim_{\K} V_{n+1} \le n+1 \,. \qedhere
\end{gather*}
\end{proof}

Secondly, we remark that $N$ as stated in Theorems~\ref{thm:tchak},~\ref{thm:quad1},~\ref{thm:quad2} may not be the minimal number $N_{\rm opt}$ of exact quadrature points possible. In fact, from the dimension of the space $V_n\subseteq L_1(\Omega,\mu;\K)$ alone one cannot deduce a lower bound on this number. 
Integration on $L_1(\Omega,\mu;\K)$ is a $\K$-linear functional. Using the notation
\begin{align*}
\Lambda_{\rm int}&: L_1(\Omega,\mu;\K) \to \K \enspace,\quad f \mapsto \int_{\Omega} f(x)\dint\mu(x) \,,
\end{align*}
we can record the following observation.

\begin{remark}%\label{rem:optN}
Each $V_n\subseteq\ker \Lambda_{\rm int}$
admits a trivial exact quadrature rule with $1$ arbitrary node and associated weight $0$.
Even infinite-dimensional spaces can thus possess such rules.
\end{remark}

The next example shows that all values $1\le N_{\rm opt} \le n$ are possible for $N_{\rm opt}$ in subspaces $V_n\subseteq L_1(\Omega,\mu;\R)$ of real-valued functions with $\dim_{\R} V_n=n$. 

\begin{example}\label{ex:Nopt}
Let $(\R,\cL,\lambda)$ be the Lebesgue measure space and $\varphi_j,\psi_j:\R\to\R$ for $j=1,\ldots,n$ the functions
\begin{align*}
\varphi_j(x):=\chi_{(j-1,j)}(x) \quad,\quad  \psi_j(x):=(x-j+\tfrac{1}{2})\chi_{(j-1,j)}(x) \,,
\end{align*}
where $\chi_{(j-1,j)}$ denotes the characteristic function of the interval $(j-1,j)$. Put for $m\in\{1,\ldots,n\}$ 
\begin{align*}
V_{m,n}:= \lin_{\R} \Big\{\varphi_j,\,\psi_k :j=1,\ldots,m;\, k=m+1,\ldots,n\Big\} \,.
\end{align*}
Then $\dim_{\R} V_{m,n}=n$, but the minimal number of nodes for exact quadrature is $N_{\rm opt}=m$.
\end{example}

The previous example can be extended to prove that all values $1\le N_{\rm opt} \le \effdim_{\R} V_n$ are possible for $N_{\rm opt}$ in exact quadrature rules with real weights in complex-valued spaces $V_n\subseteq L_1(\Omega,\mu;\C)$. 
In this sense, the number of nodes provided by Theorems~\ref{thm:tchak}, \ref{thm:quad1}, and~\ref{thm:quad2}, is sharp. However, in many relevant settings, e.g., Gaussian quadrature of polynomials, the number of required points is significantly smaller, see e.g.~\cite{Cools1997,NoHi07}. More insight into the topic of minimal node sets is provided by Proposition~\ref{prop:Nopt} in the next section.
It illustrates that $N_{\rm opt}$ depends on the structure of the zero sets of the functions in $V_n$.

%%%%%%%%%%%%%%%%%%%%%%%%%%%%%%%%%%%%%%%%%%%%%%%%%%%%
%--------------------------------------------------%
%%%%%%%%%%%%%%%%%%%%%%%%%%%%%%%%%%%%%%%%%%%%%%%%%%%%

%\newpage

\section{Discretization of positive linear functionals}

Let $\K\in\{\R,\C\}$ and assume that $V_n:=\lin_{\K}\{ \varphi_1,\ldots,\varphi_n\}$ is a linear function space spanned by linearly independent functions
\[
\varphi_1,\ldots,\varphi_n:\Omega\to\K \quad\text{ on some generic set }\Omega\neq\emptyset\,.
\]
For any $\K$-linear functional $L:V_n\to\K$ there then exist $n=\dim_{\K}V_n$ nodes $x_1,\ldots,x_n\in\Omega$ such that
\begin{align*}
L f = \sum_{j=1}^{n} \lambda_j f(x_j) \quad\text{ for all }f\in V_n
\end{align*}
with weights $\lambda_1,\ldots,\lambda_n\in\K$. This fact is well-known and analogous to Theorem~\ref{thm:quad1}.

In this section we aim for conditions, when 
$L$ is discretizable with real or non-negative weights, analogous to the real-weighted and Tchakaloff discretization of the integral in Theorems~\ref{thm:quad2} and~\ref{thm:tchak}.

\begin{definition}
A $\K$-linear functional $L:V_n\to\K$ is called \emph{discretizable with real weights} if and only if there exists a node set $P=\{x_1,\ldots,x_N \} \subseteq \Omega$ and a real weight set $W=\{\lambda_1,\ldots,\lambda_N\}\subset\R$ such that
\begin{align*}
	L f = \sum_{j=1}^{N} \lambda_j f(x_j) \quad\text{for all } f\in V_n\,.
\end{align*}	
It is called \emph{real-preserving} if and only if it maps the purely real-valued functions in $V_n$ to
real values.
\end{definition}

\begin{remark}\label{rem:real-preserving}
Real-preserving functionals $L$ are also characterized by the property that they map purely imaginary-valued functions in $V_n$ to imaginary values.
\end{remark}

\begin{lemma}\label{lem:aux_functional}
A $\K$-linear functional $L:V_n\to\K$ is real-preserving if and only if there exists an $\R$-linear functional $\tilde{L}:\Re(V_n)\to\R$, where $\Re$ is the $\R$-linear map defined in~\eqref{eqdef:Re-Op}, such that
\begin{align}\label{eqref:ass_real-functional}
\Re\circ L = \tilde{L} \circ\Re \,.
\end{align}
\end{lemma}
\begin{proof}
The imaginary-valued functions in $V_n$ constitute the kernel $\ker(\Re)$ of the $\R$-linear map $\Re:V_n\to \Re(V_n)$.
Hence, if $L:V_n\to\K$ is real-preserving,
$Lf\in\ii\R$ for each $f\in\ker(\Re)$. We conclude $\Re(Lf)=0$ and thus
\[
\ker\big(\Re\big) \subseteq \ker\big(\Re\circ L\big) \,.
\]
There is hence a uniquely determined $\R$-linear functional
\[
\tilde{L}: \Re(V_n) \to \R
\]
such that \eqref{eqref:ass_real-functional} holds true. 
Conversely, if~\eqref{eqref:ass_real-functional} holds, we deduce
for real-valued $f\in V_n$
\begin{gather*}
Lf = \Re(Lf) - \ii \Re(\ii Lf) =
\tilde{L}\Re(f) - \ii \tilde{L} \Re(\ii f) = \tilde{L}f \in \R \,. \qedhere
\end{gather*}
\end{proof}

Recall that $\Re(V_n)=\Im(V_n)$, where $\Im$ is defined as in~\eqref{eqdef:Im-Op}. 
Hence $\tilde{L}$ also operates on $\Im(V_n)$ and due to~\eqref{eqref:ass_real-functional} it holds
\begin{align}\label{eqref:ass_imaginary-functional}
\Im \circ L = \tilde{L} \circ\Im 
\end{align}
as the following calculation
%, which uses~\eqref{eqref:ass_real-functional}, 
shows,
\begin{gather*}
\Im(L f) = - \Re(L(\ii f)) = 
- \tilde{L}\Re(\ii f)=
\tilde{L}\Im(f) \,. 
\end{gather*}

The main result on real-weighted discretization is the following theorem.

\begin{theorem}\label{thm:real_functionals}
Let $V_n:=\lin_{\K}\{ \varphi_1,\ldots,\varphi_n\}$ and $L:V_n\to\K$ be a $\K$-linear functional. Then $L$ is discretizable with real weights if and only if
it is real-preserving.
\end{theorem}
\begin{proof}
The necessity is clear. For sufficiency, let $\tilde{L}$ be the functional associated to $L$ according to
Lemma~\ref{lem:aux_functional}.
Being an $\R$-linear functional, $\tilde{L}$ has a discretization with at most, recall Definition~\ref{def:eff_dim}, $N=\effdim_{\R} V_n = \dim_{\R}\Re (V_n)$ nodes and real weights. In view of~\eqref{eqref:ass_real-functional} and~\eqref{eqref:ass_imaginary-functional}, we get
\begin{gather*}
\Re(L f) = \tilde{L}\Re(f) = \sum_{j=1}^{N} \lambda_j \Re(f)(x_j) \quad\text{and}\quad
\Im(L f) = \tilde{L}\Im(f) = \sum_{j=1}^{N} \lambda_j \Im(f)(x_j) \,. \qedhere
\end{gather*}
\end{proof}

\begin{remark}
Looking into the proof of Theorem~\ref{thm:real_functionals}, we can record that for the real-weighted discretization of $L$ always $N=\effdim_{\R} V_n$ nodes and weights are sufficient.
\end{remark}

We next aim for positive discretizability.

\begin{definition}
A $\K$-linear functional $L:V_n\to\K$ is called \emph{positively discretizable} if and only if there exists a node set $P=\{x_1,\ldots,x_N \} \subseteq \Omega$ and a non-negative weight set $W_{\ge0}=\{\lambda_1,\ldots,\lambda_N\}$ such that
\begin{align}\label{eq:pos_discretization}
	L f = \sum_{j=1}^{N} \lambda_j f(x_j) \quad\text{for all } f\in V_n\,.
\end{align}	
\end{definition}

Recall that, in case $\K=\R$, an $\R$-linear functional $L:V_n\to\R$  is called \emph{positive} if for every function $f\in V_n$
\[
f\ge 0 \quad\Rightarrow\quad Lf\ge 0 \,.
\]
Positive functionals $L$ on a $1$-dimensional space $V_1$ always allow positive discretization. Either $L=0$, which is trivial, or there is $\varphi\in V_1$ with $L\varphi>0$, which implies the existence of an $x\in\Omega$ with $\varphi(x)>0$ since $\varphi\le 0$  would lead to the contradiction $L\varphi<0$. This $x$ can be used as discretization node with associated weight
\[
\lambda := \frac{L\varphi}{\varphi(x)} >0 \,.
\]
In a $2$-dimensional space the situation is already different.

\begin{example}\label{ex:pos_discr}
	Let $\Omega:=\R_{>0}$ and $V_2:=\lin_{\R}\{\varphi_1,\varphi_2\}$ with $\varphi_1(x):= 1/x$ and $\varphi_2(x):=1$. Consider the linear functional $L:V_2\to\R$ given by
	\begin{align*}
	L\varphi_1:= 1 \quad,\quad L\varphi_2:=0 \,.
	\end{align*}
	It is positive, i.e., $Lf\ge 0$ whenever $f\ge0$, but not positively discretizable.
\end{example}

In dimensions $n\ge2$, the class of positively discretizable functionals may hence be strictly smaller than the class of positive functionals. The situation is hence more involved than for the discretization of functionals by real weights, where Theorem~\ref{thm:real_functionals} gave a simple characterization. 

Before we proceed to our main theorem, Theorem~\ref{thm:pos_functionals}, let us consider another special situation, namely function spaces $V_n$ on a finite domain $\Omega=\{x_1,\ldots,x_M\}$.
Here one can in fact prove that every positive linear functional on $V_n$ is positively discretizable.

\begin{proposition}\label{prop:funct_on_finite_domain}
Any positive linear functional $L:V_n\to \R$ defined on a linear subspace $V_n$ of $\R^{\Omega}$, where $\Omega$ is finite, is positively discretizable.
\end{proposition}
\begin{proof}
Let $\varphi_1,\ldots,\varphi_n:\Omega\to\R$ be a basis of $V_n$. 
In case $n=1$ we already know, by the argument given before Example~\ref{ex:pos_discr}, that every positive functional on $V_1$ is positively discretizable. 
To prove the case $n\in\{2,\ldots,|\Omega|\}$, let us define $\bvarphi_L := (L(\varphi_1),\ldots,L(\varphi_{n}))^\top \in \R^{n}$ and the vectors
$\bvarphi(x):=(\varphi_1(x),\ldots,\varphi_{n}(x))^\top\in\R^{n}$ for each $x\in\Omega$. 
When we can show $\bvarphi_L \in \Phi$, where
$\Phi$ is the cone
\begin{align*}
\Phi := \cone \Big\{ \bvarphi(x)  : x\in \Omega \Big\} \subseteq \R^{n} \,,
\end{align*}
we are finished. The dual $P:=\Phi^\ast$ of $\Phi$ is, by definition, given by
\begin{align*}
P= \Big\{ \bc:=(c_1,\ldots,c_{n})^\top \in\R^{n} : \forall_{x\in \Omega} \langle \bvarphi(x), \bc\rangle \ge 0  \Big\} \,.
\end{align*}
This cone contains the coefficient vectors of positive functions in $V_n$ and therefore, since $L$ is positive, 
\begin{align*}
\langle \bvarphi_L , \bc \rangle \ge 0  
\quad\text{ for all $\bc\in P$}\,.
\end{align*}
This, in turn, shows $\bvarphi_L \in P^{\ast}$ and
hence there exists an exact positive discretization of $L$ due to $P^{\ast}=(\Phi^\ast)^\ast=\overline{\Phi}$ by the bidual theorem, see e.g.~\cite[pp.~121–125]{Rockafellar1997} and~\cite[pp.~51-53]{BoydVandenberghe2004}, and $\overline{\Phi}=\Phi$
since $\Omega$ is finite. 
\end{proof}

As another step towards Theorem~\ref{thm:pos_functionals},
we next generalize the concept of positivity of a functional to subsets $S\subseteq\Omega$.
 
\begin{definition}\label{def:D-positive}
Let $\emptyset\subset S\subseteq\Omega$. We call a function $f:\Omega\to\R$  \emph{$S$-positive} if and only if $f|_{S}\ge 0$.
An $\R$-linear functional $L:V\to\R$ is called \emph{$S$-positive} if and only if $f|_{S}\ge 0$ implies $Lf\ge 0$.
\end{definition}

The theorem below now gives a precise condition, when $\R$-linear functionals on finite-dimensional real-valued function spaces are positively discretizable.

\begin{theorem}\label{thm:pos_functionals}
For a non-trivial linear functional $L:V_n\to \R$ the following assertions are equivalent:
\begin{enumerate}[label=\normalfont(\roman*)]
\item 
$L$ is positively discretizable.
\item
There exists a subset $S\neq\emptyset$ of $\Omega$ such that
\begin{enumerate}[label=\normalfont(\alph*)]
\item $L$ is $S$-positive.
\item The $S$-positive functions in $\ker L$ vanish on $S$.
\end{enumerate}
\end{enumerate}
If $L$ is positively discretizable the maximal number of discretizing nodes needed is $n=\dim_{\R}V_n$.
\end{theorem}
\begin{proof}
(i)$\Rightarrow$(ii):\,
Choose points $\cP=\{x_1,\ldots,x_N\}$
and weights $\{\lambda_1,\ldots,\lambda_N\}\subset\R_{\ge0}$ such that 
\begin{align}\label{eq:aux10}
Lf=\sum_{j=1}^{N} \lambda_j f(x_j) \quad\text{ for all }f\in V_n \,.
\end{align}
The set $S:=\{ x_j \in\cP : \lambda_j> 0 ,\, j=1,\ldots,N \}$ is then non-empty since $L\neq 0$.  
Further, $L$ is clearly $S$-positive and $S$-positive functions in $\ker L$ vanish on $S$.
A Carathéodory subsampling step, see Lemma~\ref{CarSub}, allows to reduce the number of nodes in~\eqref{eq:aux10} to at most $n=\dim_{\R}V_n$.

(ii)$\Rightarrow$(i):\, This direction is proved by induction on $n$. Compare to the proof of Lemma~\ref{lem:intcone_eq_cone}.

$n=1$:\, Recall the argument, given before Example~\ref{ex:pos_discr}, that every positive functional on $V_1$ is positively discretizable. 
In case of a $S$-positive functional, this argument can be refined to show that a discretization with nodes from $S$ is possible. This is due to the fact that, under the assumption of $S$-positivity of $L$, a function $\varphi\in V_1$ with $L\varphi>0$ cannot have the property $\varphi|_S\le 0$. Hence there exists an $x\in S$ with $\varphi(x)>0$, which together with the positive weight $\lambda := L\varphi/\varphi(x)$ yields positive discretization.

$n\rightsquigarrow n+1$:\,
Let $\varphi_1,\ldots,\varphi_{n+1}$ be a basis of $V_{n+1}$ and put
$\bvarphi(x):=(\varphi_1(x),\ldots,\varphi_{n+1}(x))^\top\in\R^{n+1}$ for each $x\in\Omega$. 
Further, define 
$\bvarphi_L := (L(\varphi_1),\ldots,L(\varphi_{n+1}))^\top \in \R^{n+1}$.
For $S\subseteq\Omega$
define the cone
\begin{align*}
\Phi_S := \cone \Big\{ \bvarphi(x) : x\in S \Big\} \subseteq \R^{n+1}
\end{align*}
and its dual
\begin{align*}
P_S := \Phi^\ast_S = \Big\{ \bc:=(c_1,\ldots,c_{n+1})^\top \in\R^{n+1} : \forall_{x\in S} \langle \bvarphi(x), \bc\rangle \ge 0  \Big\} \,.
\end{align*}
It then holds $\bvarphi_L \in P_S^{\ast}$ since $L$ is assumed to be $S$-positive by condition (a), whence
\begin{align*}
\langle \bvarphi_L , \bc \rangle \ge 0  
\quad\text{ for all $\bc\in P_S$}\,.
\end{align*}
As in the proof of Proposition~\ref{prop:funct_on_finite_domain},
we now use $(\Phi^\ast_S)^\ast = \overline{\Phi}_S$.
This latter equality is the statement of the bidual theorem which can be looked up e.g.~in~\cite[pp.~121–125]{Rockafellar1997} or~\cite[pp.~51-53]{BoydVandenberghe2004}. It follows
$\bvarphi_L \in \overline{\Phi}_S$.
If $\bvarphi_L \in \Phi_S$ there exists an exact discretization of $L$ with non-negative weights and we are finished.
Otherwise, we have the case
\begin{align*}
\bvarphi_L \in \overline{\Phi}_S\backslash \Phi_S \subseteq \partial \Phi_S \,.
\end{align*}
Then there is a unit vector $\bn\in\R^{n+1}$, which points into $\Phi_S$ and defines
a supporting hyperplane $\cE\subset\R^{n+1}$ with $\bvarphi_L\in\cE$
and $\cE\cap \operatorname{int} \Phi_S=\emptyset$, see Def.~\ref{def:support_hyperplane} and Thm.~\ref{thm:support_hyperplane}.
The function
\[
g:\Omega\to\R \quad,\quad x\mapsto \langle \bvarphi(x), \bn \rangle 
\]
is a non-trivial element of $V_{n+1}$, due to $\bn\neq 0$, and satisfies
$L(g)=\langle \bvarphi_L, \bn \rangle=0$.
Further, since $\bn$ points into $\Phi_S$, $g$ is $S$-positive.
Due to condition (b) it hence follows
\begin{align}\label{eq:D_vanish}
g|_S = 0 \,.
\end{align}
We next choose functions $\psi_1,\ldots,\psi_n\in V_{n+1}$ such that 
$V_{n+1}= \lin_{\R}\{g,\psi_1,\ldots,\psi_n \}$.
By induction hypothesis there exist $x_1,\ldots,x_n$ in $S$, see the initial induction step for $n=1$, such that 
\[
L h = \sum_{j=1}^{n} \lambda_j h(x_j) \quad\text{ with non-negative weights $\lambda_1,\ldots,\lambda_n\ge 0$}
\]
for all $h\in V_{n}:= \lin_{\R}\{\psi_1,\ldots,\psi_n \}$.
Now, let $f= c g + h$ with $h\in V_n$ and $c\in\R$. Then, in view of $g(x_1)=\ldots = g(x_n)=0$ and due to~\eqref{eq:D_vanish},
\begin{gather*}
L(f) = c L(g) + L(h) = \sum_{j=1}^{n} \lambda_j h(x_j) 
= \sum_{j=1}^{n} \lambda_j (cg+h)(x_j) = \sum_{j=1}^{n} \lambda_j f(x_j) \qedhere
\end{gather*}
\end{proof}

\begin{remark}
If $L=0$, which is excluded in Theorem~\ref{thm:pos_functionals}, item (i) is  always fulfilled. However, there are cases where no $S\neq\emptyset$ with the stated properties exists. Further, $V_0:=\{0\}$ is possible as function space if $L=0$ with $\dim_{\R} V_0=0$.
\end{remark}

Note that condition~(ii) in Theorem~\ref{thm:pos_functionals} can also be stated as follows: There exists $\emptyset\subset S\subseteq\Omega$ such that $L$ is strictly $S$-positive, i.e., such that $L$ is $S$-positive and~\eqref{def:strict_D-positive} holds true.
A direct consequence of Theorem~\ref{thm:pos_functionals} is hence the following corollary. Here Theorem~\ref{thm:pos_functionals} can be applied with $S=\Omega$. 

\begin{corollary}
Any strictly positive linear functional $L:V_n\to \R$, in the sense that $f\ge 0$ and $f\neq 0$ implies $Lf>0$, is positively discretizable.
\end{corollary}

The following example gives an instance of a positive functional on a space of continuous functions $V_2$ on a compact domain $\Omega$ which is not positively discretizable.

\begin{example}\label{ex:pos_discr_2}
Let $\Omega:=[-1,1]$ and $V_2:=\lin_{\R}\{\varphi_1,\varphi_2\}$ with $\varphi_1(x):= \sgn(x)\sqrt{|x|}$ and $\varphi_2(x):=\max\{0,x\}$. The functional $L:V_2\to\R$ is given by
\begin{align*}
L\varphi_1:= 1 \quad,\quad L\varphi_2:=0 \,.
\end{align*}
It is positive, yet not positively discretizable. 
\end{example}
\begin{proof}
Let us first show that $L$ is positive. The only way to obtain positivity for a function $h:=\alpha\varphi_1 + \beta\varphi_2\in V_2$ is by choosing $\alpha\le 0$ and $\beta\ge0$. The other cases can be ruled out quickly by looking at the sign of $h$ on $[-1,0]$ and $[0,1]$. But $h\ge 0$ implies for every $x\in[0,1]$
\begin{align*}
\alpha \sqrt{x} + \beta x \ge 0 \quad\Leftrightarrow\quad  \beta x \ge -\alpha\sqrt{x} \,,
\end{align*}
which under the assumption $\beta\ge0$ can only be true if $\alpha\ge 0$. This shows $\alpha=0$ and thus that $h=\beta\varphi_2$ with $\beta\ge0$ are the only positive functions in $V_2$ and those are mapped to $0$. Hence, $L$ is positive.

Next, we apply Theorem~\ref{thm:pos_functionals} to show that no positive discretization of $L$ exists.
Assume that $S\subseteq[-1,1]$ satisfies condition (b) of (ii) in Theorem~\ref{thm:pos_functionals}. Since
$\ker(L)=\lin_{\R}\{\varphi_2\}$ and $\varphi_2|_{(0,1]}>0$, this implies $S\subseteq[-1,0]$. But for such $S$ the functional $L$ is not $S$-positive since $\varphi_1|_{S}\le0$
and $L\varphi_1>0$. 
\end{proof}

The crucial observation for an extension of Theorem~\ref{thm:pos_functionals} to $\K$-linear functionals is the following lemma.

\begin{lemma}\label{lem:pos_functionals_2}
A $\K$-linear functional $L:V_n\to \K$ is positively discretizable if and only if it is real-preserving and the associated $\R$-linear functional $\tilde{L}:\Re(V_n)\to\R$, in the sense of Lemma~\ref{lem:aux_functional}, is positively discretizable. 
The discretization nodes and weights for $\tilde{L}$ and $L$ can be chosen equal.
\end{lemma}
\begin{proof}
If $L$ is positively discretizable, there are nodes $\cP=\{x_1,\ldots,x_N\}$
and $\{\lambda_1,\ldots,\lambda_N\}\subset\R_{\ge0}$ such that~\eqref{eq:pos_discretization} holds true.
Clearly, $L$ is then real-preserving and, applying $\Re$ to~\eqref{eq:pos_discretization}, one obtains
\[
\tilde{L}\Re(f) = \Re(Lf) =\sum_{j=1}^{N} \lambda_j \Re(f)(x_j) \quad\text{ for all }f\in V_n \,.
\]
This proves the positive discretizability of $\tilde{L}$. Vice versa, assuming that $L$ is real-preserving and $\tilde{L}$ positively discretizable,
we argue with~\eqref{eqref:ass_real-functional} and~\eqref{eqref:ass_imaginary-functional} for $f\in V_n$
\begin{gather*}
L f = \Re(Lf) + \ii\Im(Lf) = \tilde{L} \Re(f) +  \ii\tilde{L}\Im(f) = \sum_{j=1}^{N} \lambda_j (\Re(f) +\ii\Im(f))(x_j) = \sum_{j=1}^{N} \lambda_j f(x_j) \,. \qedhere
\end{gather*}
\end{proof}

Taking Lemma~\ref{lem:pos_functionals_2} into account we can finally extend Theorem~\ref{thm:real_functionals} to $\K$-linear functionals. Let us start with some notation.
For the function space $V_n\subseteq\K^\Omega$, where $\Omega\neq\emptyset$ is some generic set, and a non-empty subset $S\subseteq\Omega$, we define the $\K$-linear restriction map
\begin{align*}
\Lambda_{S}&: V_n \to \K^S \,,\quad f \mapsto f|_S \,. 
\end{align*}  
Further, for an $\R$-linear map $R:V\to \R^I$ from a real vector space $V$ to $\R^I$, where $I$ is an arbitrary index set, we will use the notation $\kerplus(R)$ for the cone
\begin{align*}
   \kerplus(R) := \{ f\in V : Rf \ge 0  \} \,.
\end{align*}
This cone extends the respective kernel $\ker(R) = \{ f\in V : Rf = 0  \}$ and is relevant for us for the mappings $R=\Lambda_S$ and $R=L$.

\begin{lemma}\label{lem3}
	The condition $\kerplus\big(\Re\circ \Lambda_S\big)\subseteq\kerplus\big(\Re\circ L\big)$ 
    implies  
    $\ker\big(\Re\circ \Lambda_S\big)\subseteq\ker\big(\Re\circ L\big) $.
\end{lemma}
\begin{proof}
	Assume $f\in \ker(\Re\circ \Lambda_S)$. Then also $-f\in \ker(\Re\circ \Lambda_S)$ and hence $\{f,-f \}  \subseteq \ker(\Re\circ \Lambda_S) \subseteq	\kerplus(\Re\circ \Lambda_S)$. 
    If $\kerplus(\Re\circ \Lambda_S)\subseteq\kerplus(\Re\circ L)$ thus  $\{f,-f \} \subseteq \kerplus(\Re\circ L)$, which means  $\Re(Lf) \ge 0$ and $-\Re(Lf) \ge 0$ and implies $f\in \ker(\Re\circ L)$.
\end{proof}

Now we are ready to state the analogue of Theorem~\ref{thm:pos_functionals} for $\K$-linear functionals. Note that condition (ii) in Theorem~\ref{thm:pos_functionals} is equivalent to
\[
\kerplus\big( \Lambda_S\big)\subseteq\kerplus\big(L\big)
\quad\text{and}\quad
\kerplus\big(\Lambda_S\big)\cap \ker\big(L\big) \subseteq \ker\big(\Lambda_S\big) \,.
\]

\begin{theorem}\label{thm:pos_functionals_2}
For a non-trivial $\K$-linear functional $L:V_n\to \K$ the assertions below are equivalent:
\begin{enumerate}[label=\normalfont(\roman*)]
\item 
$L$ is positively discretizable.
\item
There exists a subset $S\neq\emptyset$ of $\Omega$ such that
\[
\kerplus\big(\Re\circ \Lambda_S\big)\subseteq\kerplus\big(\Re\circ L\big)
\quad\text{and}\quad
\kerplus\big(\Re\circ \Lambda_S\big)\cap \ker\big(\Re\circ L\big) \subseteq \ker\big(\Re\circ \Lambda_S\big) \,.
\]
\end{enumerate}
If $L$ is positively discretizable the maximal number of discretizing nodes needed is $N=\effdim_{\R}V_n$.
\end{theorem}
\begin{proof}
(i)$\Rightarrow$(ii):\,
Assuming~(i), $L$ is real-preserving and the associated $\R$-linear functional $\tilde{L}$ is positively discretizable due to Lemma~\ref{lem:pos_functionals_2}. By Theorem~\ref{thm:pos_functionals} the positive discretizability of $\tilde{L}$ translates to the conditions in (ii).

(ii)$\Rightarrow$(i):\, 
The first condition in (ii) implies $\kerplus(\Re\circ \Lambda_\Omega)\subseteq \kerplus(\Re\circ L)$ and thus, as a consequence of Lemma~\ref{lem3}, $\ker(\Re\circ \Lambda_\Omega)\subseteq \ker(\Re\circ L)$. This shows that $L$ is real-preserving, see Remark~\ref{rem:real-preserving}.
Let $\tilde{L}$ denote the associated $\R$-linear functional according to Lemma~\ref{lem:aux_functional}. 
The first condition in (ii) states that $\tilde{L}$ is $S$-positive and the second means that $S$-positive functions in $\ker \tilde{L}$ vanish on $S$. 
Theorem~\ref{thm:pos_functionals} thus yields a positive discretization of $\tilde{L}$ with at most $N=\effdim_{\R} V_n$ nodes and positive weights. This discretization
extends to $L$ according to Lemma~\ref{lem:pos_functionals_2}.
\end{proof}

We next show that Theorem~\ref{thm:tchak} on Tchakaloff quadrature can be easily deduced from Theorem~\ref{thm:pos_functionals_2}.

\begin{tcolorbox}
[ colback=white,
  colframe=black,
  arc=0mm,
  boxrule=1pt,
  fonttitle=\bfseries
]
\begin{proof}[\textbf{\textup{Proof of Theorem~\ref{thm:tchak} (based on positive discretizable functionals, Theorem~\ref{thm:pos_functionals_2})}}]

~\\[-1.2ex]
In the setting of Theorem~\ref{thm:tchak}, the integration on $V_n$ defines a $\K$-linear functional $L_{\rm int}$ and an associated positive $\R$-linear functional $\tilde{L}_{\rm int}$ acting on $\Re V_n$. 
The latter is $S$-positive for the set $S:=\Omega\backslash \cN$, where
\[
\cN:= \{ x\in\Omega ~:~ \exists  
f\in \Re V_n: f \sim_{\mu} 0  
\wedge f(x)\neq 0   \} \,,
\]
since $\cN$ is a $\mu$-null set. Further, $\tilde{L}_{\rm int}f=0$ for a function $f\in\Re V_n$ with $f|_S\ge 0$ implies $f\sim_{\mu}0$ and thus $f|_{S}=0$. 
Hence, the assumptions of Theorem~\ref{thm:pos_functionals_2}~(ii) are fulfilled and yield a Tchakaloff discretization of $L_{\rm int}$. 
\end{proof}
\end{tcolorbox}

Let us finally comment on the number of nodes needed for discretization.
We know that $n=\dim_{\K}V_n$ are always sufficient for the discretization of a $\K$-linear functional $L:V_n\to\K$ if weights maybe chosen from $\K$. If we restrict to real or non-negative weights $N=\effdim_{\R}V_n$ nodes may be needed.
In many situations, however, these numbers are too large. We therefore end this section with a closer analysis of the point sets $\cP\subseteq\Omega$ that are actually suitable for exact discretization.

\begin{definition}
We call a point set $\cP:=\{x_1,\ldots,x_N\}\subseteq\Omega$ \emph{suitable for exact discretization of $L$} if there exist associated weights $\mu_1,\ldots,\mu_N$
such that
\begin{align*}
Lf = \sum_{j=1}^{N} \mu_j f(x_j) \quad\text{ for all $f\in V_n$}\,.
\end{align*}
\end{definition}

\begin{proposition}\label{prop:Nopt}
Let $\varphi_1,\ldots,\varphi_n:\Omega\to\K$ be functions on a set $\Omega\neq\emptyset$ and 
let $L:V_n\to\K$ be a $\K$-linear functional on
$V_n:=\lin_{\K}\{\varphi_1,\ldots,\varphi_n\}$.
Put $\bvarphi_L:= (L(\varphi_1),\ldots,L(\varphi_n))^\top\in\K^{n}$ and $\bvarphi(x):=(\varphi_1(x),\ldots,\varphi_n(x))^\top\in\K^{n}$ for $x\in\Omega$. 
Suitability of a point set $\cP:=\{x_1,\ldots,x_N\}\subseteq\Omega$ for exact discretization of $L$ can be characterized as follows. 
\begin{enumerate}[label=\normalfont(\roman*)]
\item
$\cP$ is suitable for exact discretization of $L$ with weights in $\K$ if and only if one of the following two equivalent conditions is satisfied.
\begin{enumerate}[label=\normalfont(\alph*)]
    \item $\bvarphi_L \in \lin_{\K}\{\bvarphi(x_1),\ldots,\bvarphi(x_N)\}$.
    \item $\ker\Lambda_{\cP} \subseteq \ker L $.
\end{enumerate}
\item
$\cP$ is suitable for exact discretization of $L$ with weights in $\R$ if and only if one of the following two equivalent conditions is satisfied.
\begin{enumerate}[label=\normalfont(\alph*)]
     \item $\bvarphi_L \in \lin_{\R}\{\bvarphi(x_1),\ldots,\bvarphi(x_N)\}$.
     \item $\ker\big(\Re\circ\Lambda_{\cP}\big) \subseteq \ker\big(\Re\circ L\big)$.
\end{enumerate}
\item
$\cP$ is suitable for exact discretization of $L$ with weights in $\R_{\ge0}$ if and only if one of the following two equivalent conditions is satisfied.
\begin{enumerate}[label=\normalfont(\alph*)]
\item $\bvarphi_L \in \cone\{\bvarphi(x_1),\ldots,\bvarphi(x_N)\}$.
\item Either $L=0$ or there exists $\cQ\subseteq\cP$ such that
$|\cQ|=\rank(\Re\circ\Lambda_{\cQ})$
and
\begin{align}\label{cond:cone_inclusion}
\kerplus\big(\Re\circ\Lambda_{\cQ}\big) \subseteq \kerplus\big(\Re\circ L\big) \,.
\end{align}
\end{enumerate}
\end{enumerate}
\end{proposition}
\begin{proof}
The conditions (a) are obvious, as are the conditions (b) in (i) and (ii). Let us consider condition (b) in (iii). 

Necessity:
If $L\neq0$ the set $\cP$ can be subsampled  %reduced 
by Carathéodory to a still suitable point set $\cQ$ with the additional property $|\cQ|=\rank (\Re\circ\Lambda_{\cQ})$.
Since $\cQ$ is suitable for exact positive discretization, \eqref{cond:cone_inclusion} is clearly fulfilled.

Sufficiency: In case $L=0$, every finite point set $\cP\subseteq\Omega$ is suitable for exact positive discretization. In case $L\neq 0$,
let $\cQ$ be the points $\{y_{1},\ldots,y_M \}\subseteq\cP$ with $M:=|\cQ|=\rank (\Re\circ\Lambda_{\cQ})$ satisfying
\eqref{cond:cone_inclusion}.
Then $M\le\effdim_{\R}V_n$. Further,
from Lemma~\ref{lem3}, we get  $\ker(\Re\circ\Lambda_{\cQ}) \subseteq \ker( \Re\circ L)$. By condition (b) in (ii) there hence exist real weights $\lambda_1,\ldots,\lambda_M\in\R$ such that
\begin{align*}
	L f = \sum_{j=1}^{M} \lambda_j f(y_j) \quad\text{ for all $f\in V_n$}\,.
\end{align*}
Note that, due to $|\cQ|=\rank(\Re\circ\Lambda_{\cQ})$,
we have $\Re\circ\Lambda_{\cQ}(V_n)=\R^{\cQ}$ and there are functions $g_1,\ldots,g_M\in V_n$ with the property $(\Re\circ g_i)(y_j)=\delta_{ij}$ for $i,j=1,\ldots,M$. These functions are contained in 
$\kerplus(\Re\circ \Lambda_\cQ)$ and thus, as a consequence of~\eqref{cond:cone_inclusion},
\[
0 \le \Re (L g_i) = \sum_{j=1}^{M} \lambda_j (\Re\circ g_i)(y_j) = \lambda_i \quad\text{ for $i=1,\ldots,M$}\,. \qedhere
\]
\end{proof}

%%%%%%%%%%%%%%%%%%%%%%%%%%%%%%%%%%%%%%%%%%%%%%%%%%%%
%--------------------------------------------------%
%%%%%%%%%%%%%%%%%%%%%%%%%%%%%%%%%%%%%%%%%%%%%%%%%%%%

%\newpage

\section{Tchakaloff quadrature and Kolmogorov widths}
\label{sec:4}

A direct consequence of 
Theorem~\ref{thm:tchak}, taking into account
Remark~\ref{rem:finite_measure}, is a relation between the Tchakaloff quadrature widths~\eqref{eqdef:kappa_numbers} and Kolmogorov widths. It fits well with the result~\cite[Prop.~2]{Nov86} by Novak, where the discretization of linear functionals is considered with real weights. 
Let $(\Omega,\Sigma,\mu)$ be a finite measure space and $B_\mu(\Omega;\mathbb{R})$ be the collection of all properly bounded $\mu$-measurable functions on $\Omega$ with values in $\R$, normed by $\|f\|_{\infty}:=\sup_{x\in\Omega}|f(x)|$.
The $n$-th Kolmogorov width of $\cF\subseteq B_\mu(\Omega;\mathbb{R})$, where $\cF$ shall be comprised of proper functions, is then given by, see~\cite[p.~45]{DuTeUl18},
\begin{align}\label{eqdef:Kolmo_infty}
d_n(\cF)_{B_\mu(\Omega)} :=  \inf\limits_{\varphi_1,\ldots,\varphi_n} \:
\sup\limits_{f \in \cF} \:
\inf\limits_{c_1,\ldots,c_n\in\R} \:
\Big\| f - \sum\limits_{j=1}^{n} c_j \varphi_j \Big\|_{\infty} \,, 
\end{align}
where $\varphi_1,\ldots,\varphi_n:\Omega\to\R$ run through $B_\mu(\Omega;\mathbb{R})$ and are therefore elements of $L_1(\Omega,\mu;\R)$.
Further, we set 
$d_0(\cF)_{B_\mu(\Omega)} :=  \sup_{f \in \cF} \| f \|_{\infty}$. 
We obtain the following result.

\begin{theorem}\label{prop:Kolmogorov}
Let $(\Omega,\Sigma,\mu)$ be a positive measure space with $\mu(\Omega)<\infty$ and
let $\cF\subseteq B_\mu(\Omega;\mathbb{R})$ be some arbitrary subclass of $\mu$-measurable functions. Then 
\begin{align*}
\kappa^{+}_{n+1}(\cF;\mu)  \le 2\mu(\Omega) d_n(\cF)_{B_\mu(\Omega)} \quad\text{ for all $n\in\N_{0}$} \,.
\end{align*}
\end{theorem}
\begin{proof}
Let $V_n$ be an $n$-dimensional subspace of $L_1(\Omega,\mu;\R)$ and let $V_{n+1} := V_n + \lin_{\K}\{ \mathds 1_\Omega \}$, where $\mathds 1_\Omega$ is the $1$-function on $\Omega$. According to Theorem~\ref{thm:tchak}
and Remark~\ref{rem:finite_measure} there exists
an associated Tchakaloff node-weight set  $(x_j,\mu_j)_{j=1,\ldots,n+1}$ with $\sum_{j=1}^{n+1} \mu_j=\mu(\Omega)$ and 
for any $f\in\cF$ and $g\in V_n$ we have
\begin{align*}
\Big| \int_{\Omega} f(x) \dint\mu(x) -  \sum_{j=1}^{n+1} \mu_j f(x_j)   \Big|
=  \Big| \int_{\Omega} (f-g)(x) \dint\mu(x) -  \sum_{j=1}^{n+1} \mu_j (f-g)(x_j)   \Big|  
\le 2 \mu(\Omega) \|f - g \|_{\infty}  \,.
\end{align*}
Here we used $\|f\|_{L_1(\mu)} \le \mu(\Omega)\|f\|_{\infty}$ and $\sum_{j=1}^{n+1} \mu_j=\mu(\Omega)$.
Optimizing over all $n$-dimensional subspaces $V_n$, we get the stated estimate.
\end{proof}

A similar proof idea can be used to deduce
a relation for special subclasses $\cF\subseteq L_2(\Omega,\mu;\R)$, where $\cF$ is as before comprised of proper functions.
Let 
\begin{align*}%\label{order_singular}
 		\sigma_1 \geq \sigma_2 \geq \sigma_3 \geq \cdots \ge 0 
\end{align*}
be a non-increasing sequence of real numbers contained in $\ell_2(\N;\R_{\ge 0})$.
Further, let $J:=\{ j\in\N : \sigma_j>0 \}$ and
$\{\eta_j\}_{j\in J}$ be an orthonormal system of proper functions in $L_2(\Omega,\mu;\R)$
associated to the non-zero numbers $\sigma_j$.
Essentially, the considered classes $\cF$ shall be of the basic form
\begin{align*}%\label{eqdef:subclass_F}
\cF_0 := \Big\{ \sum_{j\in J} c_j \eta_j ~:~  \{c_j\}_{j\in J} \subset \R \enspace,\quad \sum_{j\in J} \Big(\frac{c_j}{\sigma_j}\Big)^2 \le 1 \Big\} \,.
\end{align*}
In addition, we allow `distortions' $\cN_f$ of the functions $f\in\cF_0$. For this, let $\cN\subseteq\Omega$ be some fixed $\mu$-null set and
for each $f\in\cF_0$ let
\begin{align*}%\label{eqdef:subclass_F}
\cN_f\subseteq \Big\{ g \text{ $\mu$-measurable} ~:~  \supp g \subseteq \cN \Big\} \,,
\end{align*}
The classes $\cF$ are then given by
\begin{align}\label{eqdef:subclass_F}
    \cF:= \Big\{ f + \cN_f~:~ f\in \cF_0 \Big\} \,.
\end{align}

\begin{remark}[Separable RKHS with finite trace] Prominent examples of such $\cF$ are the unit balls $B_1(\cH):=\{f\in\cH : \|f\|_{\cH}\le 1\}$ in separable 
reproducing kernel Hilbert spaces $\cH := \cH(k)$ having a  finite trace
\begin{equation}\label{eq:trace}
    \int_{\Omega} k(x,x)\dint\mu(x)<\infty
\end{equation}
with respect to a finite (probability) measure $\mu$. As shown in Example \ref{ex:non-separable_RKHS} below, the separability condition seems to be essential.
In this case,
the non-zero numbers $\sigma_j$ in the definition of $\cF$, indexed by $J$, correspond to the non-zero singular numbers of the embedding
\begin{equation*}%\label{Id_embed}
\Id: \cH \rightarrow L_2(\Omega,\mu;\R) \,.
\end{equation*}
Fixing associated systems of left- and right-singular functions 
$\{e_j\}_{j\in J}$ and $\{\eta_j\}_{j\in J}$, 
which shall satisfy $e_j=\sigma_j\eta_j$ and 
\begin{align*}
\langle \eta_j,\eta_k \rangle_{L_2(\mu)} = \delta_{j,k} = \langle e_j,e_k \rangle_{\cH} \,,
\end{align*}
any function $f\in B_1(\cH)$
can be represented in the form 
\begin{align*}
f = \sum_{j\in J} \langle f,e_j \rangle_{\cH} e_j  + d_f  = \sum_{j\in J} \langle f,\eta_j \rangle_{L_2(\mu)} \eta_j  + d_f 
\end{align*}
with $d_f\in\ker\Id$ and $\|\{c_j/\sigma_j\}_{j}\|_{\ell_2}\le 1$ for $c_j:= \langle f,\eta_j \rangle_{L_2(\mu)}$. Put $\cN_f:=\{d_f\}$. Then $B_1(\cH)$ is of the form~\eqref{eqdef:subclass_F}.
Indeed, the separability of $\cH$ ensures that $\ker\Id$ has a countable basis and that hence
\begin{align*}
\cN:= \{ x\in\Omega ~:~ \exists  
f\in \ker \Id: f(x)\neq 0   \} 
\end{align*}
is a $\mu$-null set. Further, $\supp d_f\subseteq\cN$ for any $f\in B_1(\cH)$. 
\end{remark}

The following bound significantly improves on the results in \cite[Thms.\ 1, 5, 6, 9]{HaObLy22}, where similar bounds under more restrictive assumptions have been stated.  

\begin{theorem}\label{prop:Kolmogorov_2}
Let $(\Omega,\Sigma,\mu)$ be a positive measure space with $\mu(\Omega)<\infty$ and
let $\cF\subseteq L_2(\Omega,\mu;\R)$ be a subclass of the form~\eqref{eqdef:subclass_F}.
Then 
\begin{align*}
\kappa^{+}_{n}(\cF;\mu) 
\leq 2\sqrt{\mu(\Omega)} \Big( \sum\limits_{j=n}^{\infty}\sigma_j^2 \Big)^{1/2} \quad\text{ for all $n\in\N$} \,.
\end{align*}
\end{theorem}
\begin{proof}
Let $D:=\Omega\setminus\cN$ denote the restricted domain, where the $\mu$-null set $\cN$ entering in the definition of $\cF$ is taken out, see~\eqref{eqdef:subclass_F}. We further denote $\tilde{\eta}_j:=\eta_j|_D$ for $j\in J$ and, if $|J|<\infty$, put $\tilde{\eta}_n:=0$ for $n>|J|$.
Put $V_0:=\{0\}$ and for $n\in\N$ let
\begin{align*}
V_n := \lin_{\R}\{ \tilde{\eta}_1,\ldots,\tilde{\eta}_n \} \,.
\end{align*}
Note that $V_n\subseteq L_1(D,\mu;\R)$ since $\mu(D)=\mu(\Omega)<\infty$.
Further, for $n\in\N_{0}$, we define the function 
\begin{align*}%\label{eqdef:E-function}
E_n(x):= \sum_{j=n+1}^\infty \sigma_j^2 \tilde{\eta}_j(x)^2 \,, 
\end{align*}
which, due to $\{\sigma_j\}_{j\in\N}\in \ell_2(\N;\R_{\ge 0})$ and either $\|\tilde{\eta}_j\|_{L_2(\mu)}=1$ or $\|\tilde{\eta}_j\|_{L_2(\mu)}=0$, is well-defined in $L_1(D,\mu;\R)$ (monotone convergence in $L_1$).  In particular, the defining series converges point-wise absolutely almost everywhere.
We adjoin
the function $\sqrt{E_n}$ to the space $V_n$, defining 
\begin{align*}
V_{n}^{+}:= V_n \oplus \lin_{\R}\{ \sqrt{E_n} \} \,.
\end{align*}
Applying Theorem~\ref{thm:tchak} to $V_{n}^{+}$ yields
an associated Tchakaloff node-weight set  $\{(x_j,\mu_j)\}_{j=1,\ldots,n+1}$ contained in $ D\times\R_{\ge0}$. 
Let $\tilde{f}:=f|_D$ denote the restriction of $f=f_0 + d\in\cF$, where $f_0=\sum_{j\in J} c_j \eta_j\in\cF_0$ and $d\in\cN_f$, see~\eqref{eqdef:subclass_F}.
Then $\tilde{f}=\sum_{j\in J} c_j \tilde{\eta}_j$ since $\supp d\subseteq\cN$ and, for $\tilde{g}:=\sum_{j=1}^{n} c_j \tilde{\eta}_j\in V_{n}^{+}$, where the coefficient sequence $\{c_j\}_{j\in J}$ is extended trivially if $n>|J|$, 
it holds 
\begin{align*}
\Big| \int_{D} \tilde{f}(x) \dint\mu(x) -  \sum_{j=1}^{n+1} \mu_j \tilde{f}(x_j)   \Big| 
\le \|\tilde{f}-\tilde{g}\|_{L_1(\mu)} + \Big| \sum_{j=1}^{n+1} \mu_j (\tilde{f}-\tilde{g})(x_j)\Big| \,.
\end{align*}
Hereby, the first term is clearly equal to
\begin{align*}
\Big| \int_{\Omega} f(x) \dint\mu(x) -  \sum_{j=1}^{n+1} \mu_j f(x_j)   \Big|
\,.
\end{align*}

Further, defining $c_j:=0$ for $j>|J|$, 
\begin{align*}
\|\tilde{f}-\tilde{g}\|^2_{L_1(\mu)}
\le \mu(\Omega) \Big\|\sum_{j=n+1}^{\infty} c_j\tilde{\eta_j}\Big\|^2_{L_2(\mu)}
= \mu(\Omega)\sum_{j=n+1}^{\infty} c^2_j
\le \mu(\Omega) \sum_{j\in J} \Big( \frac{c_j}{\sigma_j}\Big)^2 \sigma^2_{n+1}
\le \mu(\Omega) \sigma^2_{n+1} 
\end{align*}
and by Hölder, for every $x\in D$, 
\[
|(\tilde{f}-\tilde{g})(x)| \le \bigg(\sum_{j=n+1}^{\infty} \Big(\frac{c_j}{\sigma_j}\Big)^2\bigg)^{1/2} \Big(\sum_{j=n+1}^{\infty} \sigma_j^2\tilde{\eta}_j(x)^2\Big)^{1/2} \le  \Big(\sum_{j=n+1}^{\infty} \sigma_j^2\tilde{\eta}_j(x)^2\Big)^{1/2}  = \sqrt{E_n(x)} \,.
\]
A calculation yields
\begin{align*}
 \|E_n\|_{L_1(\mu)} = \sum_{j=n+1}^{\infty}\sigma_j^2 \,.
\end{align*}
We obtain, since the Tchakaloff nodes $x_j$ lie in $D$,
\begin{align*}
   \Big| \sum_{j=1}^{n+1} \mu_j (\tilde{f}-\tilde{g})(x_j)\Big| &\le  \sum_{j=1}^{n+1} \mu_j \sqrt{E_n(x_j)}  = \big\|\sqrt{E_n} \big\|_{L_1(\mu)} 
   \le  
   \sqrt{\mu(\Omega)}\Big( \sum_{j=n+1}^{\infty}\sigma_j^2 \Big)^{1/2} \,.  \qedhere
\end{align*}
\end{proof}

\begin{remark} The result in Theorem~\ref{prop:Kolmogorov_2} may be generalized to classes $\cF\subseteq L_2(\Omega,\mu;\R)$ for which $d_n(\cF)_{L_2(\mu)}/\sqrt{n}$ are summable, see \cite[Prop.\ 1]{KrPoUlUl25}. The condition \eqref{eqdef:subclass_F} should then be replaced with \cite[Assumption B]{KrPoUlUl25}. 
\end{remark}

\begin{remark}[Discussion of Theorems \ref{prop:Kolmogorov}, \ref{prop:Kolmogorov_2}] {\em (i)} The result in Theorem~\ref{prop:Kolmogorov_2} has to be compared to \cite[Thm\ 7.2]{KaUlVo21}, where arbitrary weights for the numerical integration are admitted. Here it is proved that there is a $c>1$ such that  
$$
    \kappa_{cn}(\cF;\mu)^2 \lesssim \frac{ \log(n)}{n}\sum\limits_{k\geq n} \sigma_k^2\,.
$$
The $\log$-term can be removed using the refined analysis in \cite{NaSchUl22}, \cite{DoKrUl23}. In any case, this bound is asymptotically better by a factor of $n^{-1/2}$ compared to the one in Theorem \ref{prop:Kolmogorov_2} if the singular values decay polynomially. This indicates that the bound in Theorem \ref{prop:Kolmogorov_2} might be improved significantly. We leave this as an open problem.  

{\em (ii)} To compare Theorem~\ref{prop:Kolmogorov_2} to the earlier result, Theorem~\ref{prop:Kolmogorov}, 
note that the numbers $\sigma_{n+1}$ coincide with the corresponding $n$-th Kolmogorov widths $d_n(\cF)_{L_2(\mu)}$ of $\cF$ with respect to $L_2(\Omega,\mu;\R)$, namely
$d_0(\cF)_{L_2(\mu)} :=  \sup_{f \in \cF} \| f \|_{L_2(\mu)}$ and 
\begin{align}\label{eqdef:Kolmo_two}
d_n(\cF)_{L_2(\mu)} :=  \inf\limits_{\varphi_1,\ldots,\varphi_n} \:
\sup\limits_{f \in \cF} \:
\inf\limits_{c_1,\ldots,c_n\in\R} \:
\Big\| f - \sum\limits_{j=1}^{n} c_j \varphi_j \Big\|_{L_2(\mu)} \quad\text{ for $n\in\N$} \,, 
\end{align}
where $\varphi_1,\ldots,\varphi_n:\Omega\to\R$ are arbitrary functions in $L_2(\Omega,\mu;\R)$.

Note that in contrast to $d_n(\cF)_{L_2(\mu)}$ the width $d_n(\cF)_{B_\mu(\Omega)}$ is usually difficult to estimate. A prominent example with surprisingly fast decaying Kolmogorov widths $d_n(\cF)_{B_\mu(\Omega)}$ is given by functions with bounded mixed derivative on the $d$-torus $H^r_{\text{mix}}(\mathbb{T}^d)$ with $r>1/2$. Here we have 
$$
    d_n(H^r_{\text{mix}}(\mathbb{T}^d))_{C} \lesssim n^{-r}(\log n)^{(d-1)r+1/2},
$$
whereas 
$$
    d_n(H^r_{\text{mix}}(\mathbb{T}^d))_{L_2} \asymp 
    n^{-r}(\log n)^{(d-1)r}\,,
$$
see \cite[Sect.\ 4.3]{DuTeUl18} and \cite{TeUl21}. Hence, for this example the bound in Theorem \ref{prop:Kolmogorov_2} is significantly worse than the one in Theorem \ref{prop:Kolmogorov}. On the other hand, the upper bound in Theorem \ref{prop:Kolmogorov} is close to the one coming from \cite[Thm\ 7.2]{KaUlVo21}, where arbitrary integration weights are used.  Finally, it is well-known that, see \cite{GoSuYo17},   
$$
\kappa^{+}_{n}(H^r_{\text{mix}}(\mathbb{T}^d);\lambda^d) \asymp n^{-r}(\log n)^{(d-1)/2}\,,
$$
holds. This shows that all the mentioned bounds are not sharp for this example.
\end{remark}

We are now discussing the following negative result for a non-separable reproducing kernel Hilbert space.

\begin{example}\label{ex:non-separable_RKHS}
We first recall the non-separable space
\begin{align*}
\ell_2([0,1];\R) :=\Big\{ f:\Omega\to\R ~:~ \supp f \text{ is countable.} \Big\}
\end{align*}
with inner product
\begin{align}\label{def:inn_product_1}
\langle f,g \rangle_{\ell_2} := \sum_{x\in\supp f\cap\supp g} f(x)g(x) \,.
\end{align}
From this space, we can build the Hilbert function space
\begin{align*}
 \cH :=   \ell_2([0,1];\R) \oplus_{\perp} \lin_{\R}\{ \mathds 1_{[0,1]} \} \,,
\end{align*}
whose inner product $\langle\cdot,\cdot\rangle_{\cH}$ is obtained by extending~\eqref{def:inn_product_1} via the assignment $\langle \mathds 1_{[0,1]},\mathds 1_{[0,1]} \rangle_{\cH}:=1$ and $\langle \mathds 1_{[0,1]},g\rangle_{\cH}:=0$ for every $g\in \ell_2([0,1];\R)$.
It has the reproducing kernel
\[
k(x,y):= \mathds{1}_{x}(y) + \mathds 1_{[0,1]}(y)  = 1 + \delta_{x,y} \enspace,\quad (x,y)\in[0,1]^2\,.
\]
The identity map $\Id: \cH \rightarrow L_2([0,1],\lambda;\R)$, where $\lambda$ is the Lebesgue measure on $[0,1]$, is a rank-$1$ embedding with $\ker\Id = \ell_2([0,1];\R)$ and $\sigma_1=1$ the only non-vanishing singular number.
Hence, if Theorem~\ref{prop:Kolmogorov_2} were applicable to 
$\cF:=B_1(\cH)\subset L_2([0,1],\lambda;\R)$, exact positive quadrature rules for $\cF$ with $n\ge 2$ nodes would exist. However, this is not possible since such rules $(x_j,\lambda_j)_{j=1,\ldots,n}$  
necessarily satisfy $\sum_{j=1}^{n} \lambda_j =1$ and thus yield quadrature $1$ for the function $\sum_{j=1}^{n} \mathds 1_{x_j}$, whose integral is $0$. 
\end{example}

The final result of this section shows that we indeed see a decay of the Tchakaloff quadrature widths for general reproducing kernel Hilbert spaces with finite trace. The proof strategy is due to \cite[Thm.\ 10.4]{NoWoII}, see also \cite[Thm.\ 1]{HaObLy22}. Note, that the argument  below shows that the separability of the space $\cH(k)$ is not required.  

\begin{theorem} Let $\mathcal{H}(k)$ denote a not necessarily separable reproducing kernel Hilbert space on $\Omega$ with positive definite kernel $k:\Omega \times \Omega \to \mathbb{R}$. Let further $\mu$ denote a probability measure on $\Omega$ such that the kernel $k$ is $\mu$-measurable and has finite trace, i.e., \eqref{eq:trace} holds true. Then it holds for any $n\in \mathbb{N}$
$$
    \kappa_n^+(\mathcal{H};\mu) \leq \frac{C_{k,\mu}}{\sqrt{n}}
$$
for some constant $C_{k,\mu}>0$ depending only on the kernel $k$ and the measure $\mu$\,.
\end{theorem}

\begin{proof} The proof goes literally along the lines of \cite[Thm.\ 10.4]{NoWoII}. The separability is not needed. Indeed, we may directly deduce from \eqref{eq:trace} that the integration functional is continuous on $\cH$. We have 
$$
    |I(f)| = \Big|\int_\Omega f(x)\dint\mu(x)\Big| \leq 
    \int_\Omega |\langle f, k(\cdot, x)\rangle_{\cH}|\dint\mu(x)
    \leq \|f\|_{\cH} \int_\Omega \sqrt{k(x,x)}\dint\mu(x)\,.
    $$
    Since $\mu$ is a probability measure we have by the Cauchy-Schwarz inequality 
    $$
        \int_\Omega \sqrt{k(x,x)}\dint\mu(x) \leq \Big(\int_\Omega k(x,x)\dint\mu(x)\Big)^{1/2}<\infty\,.
    $$
    Let us now consider the integration error functional $$R_n(f) = I(f) - \frac{1}{n}\sum\limits_{i=1}^n f(x_i)\,.$$
    Clearly, for fixed $x_1,\ldots,x_n$ we have 
    \begin{equation*}
       \begin{split} 
        \kappa_n^+(\mathcal{H},\mu)^2 &\leq \|R_n\|_{\mathcal{H}'}^2\\
        & = \int_\Omega \int_\Omega k(x,y)\dint\mu(x)d\mu(y)-\frac{2}{n}\sum\limits_{i=1}^n \int_\Omega k(x_i,y)\dint\mu(y)+\frac{1}{n^2}\sum\limits_{i=1}^n\sum\limits_{j=1}^n k(x_i,x_j)\,.
       \end{split} 
    \end{equation*}
    Averaging over all $x_1,\ldots,x_n$, i.e., integrating $n$ times with respect to $\mu$ yields
    \begin{equation*}
       \begin{split}
        \kappa_n^+(\mathcal{H},\mu)^2 &\leq 
        \Big(1-2+1-\frac{1}{n}\Big)\int_\Omega \int_\Omega k(x,y)\dint\mu(x)d\mu(y)+\frac{1}{n}\int_\Omega
        k(x,x)\dint\mu(x)\\
        &= \frac{\int_{\Omega} k(x,x)\dint\mu(x) - \int_{\Omega} \int_{\Omega} k(x,y)\dint\mu(x)d\mu(y)}{n}\\
        & = \frac{C^2_{k,\mu}}{n}\,. \qedhere
      \end{split} 
    \end{equation*}
\end{proof}

The previous result does not only show the existence of a  quadrature rule with positive weights and the stated properties. It rather shows that 
$$
    \Ept_{x_1\sim \mu,\ldots,x_n\sim \mu} \bigg(\sup\limits_{\|\cH\|\leq 1}\Big|\int_\Omega f(x)\dint\mu(x)- \frac{1}{n}\sum\limits_{i=1}^n f(x_i)\Big|\bigg) \leq \frac{C_{k,\mu}}{\sqrt{n}}\,.
$$
The following result shows that a slightly weaker  result holds with overwhelming probability if $n$ is moderately large. 
\begin{remark}[Quadrature with high probability] Using the technique in \cite[Thm.\ 7.1~(ii)]{MoUl21} and the fact that 
$$
    \|f+g\|_{L_2}^2-\|f-g\|_{L_2}^2 = 4\langle f,g\rangle_{L_2}
$$
we can even prove a high probability version of the result above by paying an extra $\log$-term in the bound. Namely, choosing $x_1,\ldots,x_n$ iid at random with respect to the density
$$
    \varrho(x) := \frac{k(x,x)}{\int_\Omega k(x,x)\dint\mu(x)}\quad, \quad x\in \Omega\,,
$$
then we have for all $n \geq n_{k,\mu}$ the bound
$$
 \Prob_{x_1\sim \varrho\dint\mu,...,x_n\sim \varrho\dint\mu}\bigg(  \sup\limits_{\|f\|_{\cH}\leq 1} \Big|\int_\Omega f(x)\dint\mu(x) - \frac{1}{n}\sum\limits_{i=1}^n\frac{f(x_i)}{\varrho(x_i)}\Big| > 
    C'_{k,\mu}\sqrt{r\frac{\log n}{n}} \bigg) \le 4n^{1-r}\,,
$$
where $r$ is chosen greater than $1$ and fixed. 
This result has to be compared to the second statement in \cite[Thm.\ 10.4]{NoWoII}, where a straight-forward concentration bound is deduced from Chebychev's inequality together with the bound on the expectation. 
\end{remark}

%%%%%%%%%%%%%%%%%%%%%%%%%%%%%%%%%%%%%%%%%%%%%%%%%%%%
%--------------------------------------------------%
%%%%%%%%%%%%%%%%%%%%%%%%%%%%%%%%%%%%%%%%%%%%%%%%%%%%

\section{Marcinkiewicz-Zygmund equalities}
\label{sec:3}

In this section we use Theorem~\ref{thm:tchak} to deduce $L_p$-MZ equalities
as in~\eqref{eqdef:LpMZequalities} in the most general conceivable setup for even exponents $p\in 2\N$.
We start with yet another lemma on the effective dimension.

\begin{lemma}\label{lem:dimensionMZ}
Let $n\in\N$, $\K\in\{\R,\C\}$, and 
$V\subseteq\K^n$ be a $\K$-vector space of finite dimension. 
Then
 \begin{align*}
 \effdim_{\R}V = \dim_{\K}V \quad\Leftrightarrow\quad V = \overline{V} \,.
 \end{align*}
 Here $\overline{V}$ denotes the $\K$-vector subspace of $\K^n$ obtained from $V$ via element-wise conjugation.
\end{lemma}
\begin{proof}
The statement is clear if $\K=\R$. To prove the case $\K=\C$, note that in this case $\Re(V)= \Im(V)$.
Further, $\widetilde{V}:=\Re(V) \oplus \ii\Im(V) \supseteq V$ and $\dim_{\C}\widetilde{V} = \dim_{\R}\Re(V) = \effdim_{\R}V$.  
Assuming $V = \overline{V}$, we then have $\Re(V)\subseteq V$ and $\ii\Im(V)=\ii\Re(V)\subseteq V$, whence $\widetilde{V}=V$, which implies $\effdim_{\R}V = \dim_{\C}V$. On the other hand, assuming the latter equality of dimensions, we have $\dim_{\C}\widetilde{V} = \dim_{\C}V$. Due to $\widetilde{V} \supseteq V$, this yields $\widetilde{V} = V$ and thus 
$\overline{V}=\overline{\widetilde{V}}=\widetilde{V}=V$.
\end{proof}

The subsequent result for the case $p=2$ is a direct consequence of Theorem~\ref{thm:tchak} and Lemma~\ref{lem:dimensionMZ}.

\begin{theorem}\label{thm:MZ_product_of_functions} %%%%%%%%%%%%%%%%%%%%%%%%%%%%%%%%%
    Let $(\Omega,\Sigma,\mu)$ be a positive measure space and $V_n$ an $n$-dimensional subspace of  $L_2(\Omega,\mu;\K)$. Then there exist $N=\dim_{\K}\lin_{\K}\{f\,\overline{g} : f,g\in V_n\}$ points $x_1,\ldots,x_{N}\in \Omega$ together with associated non-negative weights $\mu_1,\ldots,\mu_{N}$  
    such that
    \begin{equation}\label{statement3_1}
        \int_\Omega f(x) \overline{g(x)} \dint\mu(x)
        = \sum_{j=1}^{N} \mu_j f(x_j) \overline{g(x_j)} \quad\text{ for all $f,g\in V_n$}\,.
    \end{equation}
    In case $\K=\R$ we have $N\leq \binom{n+1}{2}$, in case $\K=\C$ we have $N\leq n^2$.
\end{theorem} %%%%%%%%%%%%%%%%%%%%%%%%%%%%%%%%%%%%%%%%%%%%%%%%%%%%%%%%%%%%%%%
\begin{proof} 
Define $V_n^{(2)}:=\lin_{\K}\{f\,\overline{g} : f,g\in V_n\}\subseteq L_1(\Omega,\mu;\K)$ and note that $V_n^{(2)}$ satisfies $V^{(2)}_n=\overline{V}^{(2)}_n$. Applying Theorem~\ref{thm:tchak} to $V_n^{(2)}$ yields $N=\effdim_{\R}V_n^{(2)}$ points $x_1,\ldots,x_N$ and weights $\mu_1,\ldots,\mu_N$ such that 
\[
\int_{\Omega} h(x) \dint\mu(x) = \sum_{j=1}^{N} \mu_j h(x_j)
\]
holds for all $h\in V_n^{(2)}$. In view of Lemma~\ref{lem:dimensionMZ}, further $N=\dim_{\K}\lin_{\K}\{f\,\overline{g} : f,g\in V_n\}$. The assertion follows.
The upper bounds on $N$ in case $\K=\R$ or $\K=\C$ are obvious.  
\end{proof} %%%%%%%%%%%%%%%%%%%%%%%%%%%%%%%%%%%%%%%%%%%%%%%%%%%%%%%%%%%%%%%%%%%

The corollary below 
is immediate. It provides $L_2$-MZ equalities in full generality for finite-dimensional subspaces of $L_2(\Omega,\mu;\K)$.

\begin{corollary}\label{cor:exactmz} %%%%%%%%%%%%%%%%%%%%%%%%%%%%%%%%%%%%%%%%%%%%%%%%
     Let $(\Omega,\Sigma,\mu)$ be a positive measure space and $V_n\subseteq L_2(\Omega,\mu;\K)$ with $\dim_{\K} V_n = n$. 
     Then there exist 
     $N$ points $x_1,\ldots,x_{N}\in\Omega$, with $N$ as in Theorem~\ref{thm:MZ_product_of_functions}, together with non-negative weights $\mu_1,\ldots,\mu_{N}$
     such that 
    \begin{equation}\label{exactmz_formula}
        \int_\Omega |f(x)|^2 \dint\mu(x) = \sum_{j=1}^{N} \mu_j |f(x_j)|^2
        \quad\text{ for all $f\in V_n$}
         \,.
    \end{equation}
\end{corollary} %%%%%%%%%%%%%%%%%%%%%%%%%%%%%%%%%%%%%%%%%%%%%%%%%%%%%%%%%%%%%%%%%

\begin{remark}
Theorem~\ref{thm:MZ_product_of_functions} and Corollary~\ref{cor:exactmz} are  
equivalent statements. By polarization, it is possible to 
retrieve Theorem~\ref{thm:MZ_product_of_functions} from Corollary~\ref{cor:exactmz}, see e.g.~\cite[Cor.~4.2]{BKPSU24}.  
\end{remark}

We finally deduce $L_p$-MZ equalities for even $p\in2\N$ with the same technique as in~\cite[Thm~3.1]{DaPrTeTi19} and~\cite[Sec.~4]{BKPSU24}.
We employ a basis of $V_n = \lin_{\K}\{\varphi_1, \dots, \varphi_n\}$ and define the space
    \begin{equation}\label{eqdef:W_n}
        W_n
        := \lin_{\K}\Big\{ \varphi_1^{k_1} \cdots \varphi_n^{k_n} : k_1,\ldots,k_n\in \mathds N_0,\, k_1+\ldots+k_n = \frac{p}{2} \Big\} \subseteq L_2(\Omega,\mu;\K) \,.
    \end{equation}
    Counting the number of weak compositions of the integer $\tfrac{p}{2}$ into $n$ parts, we get for the dimension
    \begin{align*}
        \dim_{\K}W_n
        \le \Big| \Big\{ k_1,\ldots,k_n\in \mathds N_0 : \sum_{j=1}^{n} k_j = \frac{p}{2} \Big\} \Big| 
        = \binom{n+\frac{p}{2}-1}{\frac{p}{2}} \,.
    \end{align*}
    With this we are ready to state and prove the main theorem.
    
\begin{theorem}\label{even_MZ} %%%%%%%%%%%%%%%%%%%%%%%%%%%%%%%%%%%%%%%%%%%%%%
    Let $(\Omega,\Sigma,\mu)$ be a positive measure space and $V_n\subseteq L_p(\Omega,\mu;\K)$ with $\dim_{\K} V_n = n$ and $p$ even.
    Let further $W_n$ be the space in~\eqref{eqdef:W_n}. Then there exist $N=\dim_{\K}\lin_{\K}\{f\,\overline{g} : f,g\in W_n\}$ points $x_1,\ldots,x_{N}\in\Omega$ and non-negative weights $\mu_1,\ldots,\mu_{N}$ such that
    \begin{equation}\label{eq:exactLpMZ}
        \int_\Omega |f(x)|^p \dint\mu(x) = \sum_{j=1}^{N} \mu_j |f(x_j)|^p
        \quad\text{ for all $f\in V_n$}\,.
    \end{equation}
In case $\K=\R$ we have $N\leq \binom{n+p-1}{p}$,  if $\K=\C$ we have $N\leq \binom{n+\frac{p}{2}-1}{\frac{p}{2}}^2$.
\end{theorem} %%%%%%%%%%%%%%%%%%%%%%%%%%%%%%%%%%%%%%%%%%%%%%%%%%%%%%%%%%%%%%%

\begin{proof} %%%%%%%%%%%%%%%%%%%%%%%%%%%%%%%%%%%%%%%%%%%%%%%%%%%%%%%%%%%%%%%%%
     Applying Corollary~\ref{cor:exactmz} to $W_n$ gives $N=\dim_{\K}\lin_{\K}\{f\,\overline{g} : f,g\in W_n\}$ points $x_1,\ldots,x_{N}\in\Omega$ and non-negative weights $\mu_1,\ldots,\mu_{N}$ such that
    \begin{equation*}
        \int_\Omega |f(x)|^2 \dint\mu(x) = \sum_{j=1}^{N} \mu_j |f(x_j)|^2
        \quad\text{ for all $f\in W_n$}
         \,.
    \end{equation*}
    In particular, this equation holds for any $f^{p/2}$ with $f\in V_n$, which is the assertion. 

    Let us put $W_n^{(2)}:=\lin_{\K}\{f\,\overline{g} : f,g\in W_n\}$.
    In case $\K=\C$, we estimate 
    \[
    \dim_{\C}W_n^{(2)} \le (\dim_{\C}W_n)^2   \le  \binom{n+\frac{p}{2}-1}{\frac{p}{2}}^2 \,.
    \]
    In case $\K=\R$, we have the bound
    \begin{align*}
        \dim_{\R}W^{(2)}_n
        &= \dim_{\R}\lin_{\R}\Big\{ \varphi_1^{k_1} \cdots \varphi_n^{k_n} : k_1,\ldots,k_n\in \mathds N_0,\, \sum_{j=1}^{n} k_j = p \Big\} \\
        &\le \Big| \Big\{ k_1,\ldots,k_n\in \mathds N_0 : \sum_{j=1}^{n} k_j = p \Big\} \Big| 
        = \binom{n+p-1}{p} \,. \qedhere
    \end{align*}
\end{proof} %%%%%%%%%%%%%%%%%%%%%%%%%%%%%%%%%%%%%%%%%%%%%%%%%%%%%%%%%%%%%%%%%%%

Utilizing Remark~\ref{rem:finite_measure}, we get companion results with an extra node where the weights sum up to $1$.

\begin{remark}
If $\mu(\Omega)<\infty$, allowing $N+1$ nodes and weights in the formulas~\eqref{statement3_1}, \eqref{exactmz_formula}, or~\eqref{eq:exactLpMZ}, it is possible to 
ensure the extra condition $\mu_1+\ldots+\mu_{N+1} = \mu(\Omega)$.
\end{remark}

We end with a comment on the actual number of evaluation points needed for MZ equalities.

\begin{remark}
The number $N$ given in Theorems~\ref{thm:MZ_product_of_functions} and~\ref{even_MZ} or in Corollary~\ref{cor:exactmz}
is usually just an upper bound for the actual number of evaluations needed. 
On the lower end, the number of evaluations can never be below $n=\dim_{\K}V_n$. The space $V_{n,0}$ from Example~\ref{ex:Nopt} is an example where this minimal number of nodes suffices.
\end{remark}

%%%%%%%%%%%%%%%%%%%%%%%%%%%%%%%%%%%%%%%%%%%%%%%%%%%%
%--------------------------------------------------%
%%%%%%%%%%%%%%%%%%%%%%%%%%%%%%%%%%%%%%%%%%%%%%%%%%%%

%\newpage

\section{Exact weighted frame subsampling}
\label{sec:5}
We begin this section with a profound result on the existence of positive measures such that a given finite system of functions becomes orthonormal. It can be immediately drawn from Proposition~\ref{prop:ContCarSub} and should be compared with~\cite[Thm.~3.1]{BKPSU24}, see also Remark~\ref{rem:discr_orth}.

\begin{proposition}\label{prop:discr_orth} %%%%%%%%%%%%%%%%%%%%%%%%%%%%%%%%%%%
    Let $\Omega$ be a set and $\varphi_1,\dots,\varphi_n\colon \Omega\to\K$ functions. Let further $N 
    = \dim_{\K}\lin_{\K}\{\varphi_k\overline{\varphi_l}:\, k,l=1,\ldots,n \}$. The following assertions are equivalent.
    \begin{enumerate}[label=\normalfont(\roman*)]
    \item
        We have
        \begin{equation*}
            \bI_n \in \cone \Big\{
            (\varphi_k(x)\overline{\varphi_l(x)}
            )_{k,l=1}^{n}
            \colon\,x\in \Omega \Big\}\,.
        \end{equation*}
     \item
        There exists a positive measure $\mu$ and an associated $\sigma$-algebra $\Sigma$ of measurable sets in $\Omega$ such that the system
        $\{\varphi_j\}_{j=1}^{n}$
        is orthonormal with respect to $\mu$, i.e.\
        \begin{equation}\label{eq:ortho_cond}
            \int_\Omega \varphi_k(x)\overline{\varphi_l(x)} \dint\mu(x)
            = \delta_{k,l}\,, \quad k,l=1,\ldots, n\,.
        \end{equation}       
    \item
        There exist $N$ points $x_1,\ldots,x_{N}\in \Omega$  and non-negative weights $\mu_1,\ldots,\mu_{N}$ such that
            $\{\varphi_j\}_{j=1}^{n}$ is orthonormal with respect to the discrete measure $\sum_{j=1}^{N} \mu_j \delta_{x_j}$, i.e.
            \begin{equation*}
                \sum\limits_{j=1}^{N} \mu_j\varphi_k(x_i)\overline{\varphi_l(x_j)}
                = \delta_{k,l}\,, \quad  k,l=1,\ldots, n \,.
            \end{equation*}
    \end{enumerate}
\end{proposition} 
\begin{proof}
The implication (i) $\Rightarrow$ (ii) is clear since, under assumption (i), there even exists a discrete positive measure $\mu$ such that \eqref{eq:ortho_cond} holds true. Concretely, we can choose $\mu$ as the discrete measure $\sum_{j=1}^{M} \nu_j\delta_{x_j}$, where the points $x_1,\ldots,x_M\in\Omega$ and weights $\nu_1,\ldots,\nu_M\ge 0$ come from the conic expansion $\bI_n=\sum_{j=1}^{M} \nu_j (\varphi_k(x_j)\overline{\varphi_l(x_j)})_{k,l=1}^{n}$. The $\sigma$-algebra $\Sigma$ of measurable sets is here the power set $\cP(\Omega)$.

To deduce (ii) $\Rightarrow$ (iii) we can invoke
Proposition~\ref{prop:ContCarSub}~(i), the conic version of Carathéodory-Tchakaloff subsampling.

(iii) $\Rightarrow$ (i) is clear by the definition of the cone.
\end{proof}

\begin{remark}\label{rem:discr_orth}
If in Proposition~\ref{prop:discr_orth}~(i) the convex hull 
\[
\conv \Big\{
            (\varphi_k(x)\overline{\varphi_l(x)}
            )_{k,l=1}^{n}
            \colon\,x\in \Omega \Big\}
\]
is used instead of the conic hull
the measure $\mu$ in (ii) can be realized as a probability measure and in (iii), allowing
$N+1$ points and weights, the additional condition $\mu_1 + \ldots + \mu_{N+1}=1$ can be fulfilled, using the convex version of Carathéodory-Tchakaloff subsampling, Proposition~\ref{prop:ContCarSub}~(ii).
\end{remark}

In the remainder of this section we will analyze the ramifications of Proposition~\ref{prop:discr_orth} for the discretization and scalability of frames in finite-dimensional Hilbert spaces. We will use the following definition.

\begin{definition}\label{def:mu-frame}
Let $\cH$ be a finite-dimensional Hilbert space
and $(\Omega,\Sigma,\mu)$ a positive measure space.
We call a measurable 
function $\bvarphi:\Omega\to\cH$ a
\emph{$\mu$-frame} for $\cH$ if there exist constants $0<A\le B<\infty$ such that
\begin{align}\label{K1FS}
A \|f\|^2 \le \int_\Omega |\langle f,\bvarphi(x) \rangle|^2 \dint\mu(x) \le B \|f\|^2 \quad\text{ for all }f\in\cH \,.
\end{align}
The largest $A$ and smallest $B$ satisfying~\eqref{K1FS} is called the \emph{lower} and \emph{upper frame bound}, respectively.
A $\mu$-frame is called a \emph{Parseval $\mu$-frame} if the constants in~\eqref{K1FS} can be chosen as $A=B=1$.
\end{definition}

Note that any $\mu$-frame $\bvarphi$ is necessarily a function in $L_2(\Omega,\mu;\cH)$, see~\eqref{eqdef:Lp_space}. Otherwise, the frame condition \eqref{K1FS} cannot hold for all $f\in\cH$.

\begin{remark}
For a $\mu$-frame $\bvarphi:\Omega\to\cH$ we will often write $\bvarphi_x$ for the vector $\bvarphi(x)\in\cH$ and then, accordingly, use the notation $\{\bvarphi_x\}_{x\in\Omega}$ for the same $\mu$-frame.
\end{remark}

There is an abundance of literature on frames, see e.g.\ the books~\cite{CasaKuty_12} or \cite{OleChrist2016} for an introductory exposition. First introduced by Duffin and Schaeffer~\cite{duffin1952class} as discrete systems in Hilbert space generalizing Riesz bases, the concept was later extended by Ali, Antoine, and Gazeau~\cite{AAG93} 
(and independently by Kaiser~\cite{Kaiser_1994}) to also allow continuous versions.
Our Definition~\ref{def:mu-frame} above of a $\mu$-frame
reproduces the notion of a continuous frame from~\cite[Def.~1.2]{FS19} 
in a finite-dimensional setting. 

Frames in finite dimensions are
usually referred to as \emph{finite-dimensional frames}.
In particular, this notion comprises 
\emph{finite frames}, i.e.\ frames consisting of finitely many elements,
due to the fact that finite systems can only be frames in spaces of finite dimension. 
A $\mu$-frame is considered \emph{discrete} if $\mu$ is such that there is a countable set $S\subseteq\Omega$ with $\mu(\Omega\setminus S)=0$. In this sense,
$\mu$-frames with a countable index $\Omega$ are particular instances of discrete $\mu$-frames. Discrete frames in the classical sense, as defined by Duffin and Schaeffer, correspond to $\mu$-frames with counting measure $\mu$ and countable index set $\Omega$.

We will later deal with $\mu$-frames, where $\mu$ is a discrete measure on $\Omega$. Recall that a positive measure $\mu$ on $\Omega$ is considered \emph{discrete} if there is an at most countable set $S\subseteq\Omega$, of which each singleton is $\mu$-measurable, such that $\mu(\Omega\setminus S)=0$. These kind of $\mu$-frames are hence also instances of discrete $\mu$-frames.

An important notion, we will need in this section, is the \emph{scalability} of frames. It
was introduced in~\cite{KUTYNIOK20132225} for classical frames.
More results on scalability can be found in~\cite{Cahill_Chen_2013, Kutyniok_Okoudjou_Philipp,Casazza_Chen_2017,CasaCarliTran23}. 
We define here the notion of
\emph{discrete scalability} for $\mu$-frames.

\begin{definition}\label{def:discr_scalable}
We call a $\mu$-frame $\bvarphi:\Omega\to\cH$ \emph{discretely scalable} if there exist nodes $\{x_j\}_{j\in\N}$ and scalars $\{s_j\}_{j\in\N}$ such that $\{s_j\bvarphi(x_j)\}_{j\in\N}$ is a discrete Parseval frame for $\cH$.
\end{definition}

\begin{remark}
One can restrict $\{s_j\}_{j\in\N}$ to non-negative scalars in Definition~\ref{def:discr_scalable} without changing the notion.
\end{remark}

We next recall some basics from frame theory.
For a given $\mu$-frame $\bvarphi$, the operator $T_{\bvarphi}$ mapping vectors $\bv\in\cH$ to the function $x\mapsto \langle \bv,\bvarphi(x)\rangle$ is called \emph{analysis operator} of $\bvarphi$.
It is a mapping from $\cH$ to $L_2(\Omega,\mu;\K)$.
Its Hilbert adjoint $T_{\bvarphi}^\ast(\mu):L_2(\Omega,\mu;\K)\to\cH$ in the space $L_2(\Omega,\mu;\K)$ is called the \emph{synthesis operator} of $\bvarphi$ with respect to $\mu$. 
It depends on $\mu$ and can be represented by the integral
\begin{align*}
T^\ast_{\bvarphi}(\mu) f = \int_{\Omega}  f(x)\bvarphi(x)  \dint\mu(x) \quad,\quad\, f\in L_2(\Omega,\mu;\K) \,.
\end{align*}
Due to Hölder, the integrand $x\mapsto f(x)\bvarphi(x)$ is guaranteed to belong to $L_1(\Omega,\mu;\cH)$.
The \emph{frame operator} $S_{\bvarphi}(\mu):=T^\ast_{\bvarphi}(\mu)\circ T_{\bvarphi}$ is the composition of the synthesis and the analysis operator and also depends on the choice of $\mu$. It is an element of $\cL(\cH)$, the space of bounded linear automorphisms of $\cH$.
In addition, it is positive definite Hermitian and has the integral representation
\begin{align*}
S_{\bvarphi}(\mu)   = \int_{\Omega}  \bvarphi(x)\bvarphi(x)^\ast  \dint\mu(x) \,.
\end{align*}
It is well-known that a $\mu$-frame is Parseval if and only if $S_{\bvarphi}(\mu)$ acts as identity $\operatorname{Id}_{\cH}$ on $\cH$.
In general, for a $\mu$-frame $\bvarphi$ with frame bounds $A$ and $B$, it holds
\begin{align}\label{eq:spectral_frameop}
A \operatorname{Id}_{\cH} \le  S_{\bvarphi}(\mu) \le B \operatorname{Id}_{\cH} \,,
\end{align}
where `$\le$' means here the 
Loewner order on positive semi-definite Hermitian operators, i.e.
\[
P_1 \le P_2 \quad\Leftrightarrow\quad P_2-P_1 \,\in\, \mathcal{C}_{\ge0}:=\big\{ Q \in\cL(\cH) ~:~ Q^\ast=Q \,,\: Q\text{ positive semi-definite} \big\} \,.
\]

The following lemma is well-known, see e.g.~\cite[Thm.~2.2]{CasaCarliTran23}
for a discrete version.

\begin{lemma}\label{lem:ortho_equiv_Pars}
Let $\bvarphi:\Omega\to\K^n$ be a map with component functions $\varphi_1,\ldots,\varphi_n:\Omega\to\K$.
For a measure $\mu$ on $\Omega$, the following properties are equivalent:
\begin{enumerate}[label=\normalfont(\roman*)]
\item The system $\{\varphi_j\}_{j=1}^{n}$ is orthonormal with respect to $\mu$.
\item
   The family $\{\bvarphi(x)\}_{x\in\Omega}$ is a Parseval $\mu$-frame in $\K^n$.
\end{enumerate}
\end{lemma}
\begin{proof}
Let $\bI_n$ denote the unit matrix in $\K^{n\times n}$, canonically representing the identity operator $\operatorname{Id}_{\K^n}$. According to the above discussion, both conditions (i) and (ii) are equivalent to
\begin{gather*}
S_{\bvarphi}(\mu) =  \int_{\Omega}  \bvarphi(x)\bvarphi(x)^\ast  \dint\mu(x) = \bI_n \,. \qedhere
\end{gather*} 
\end{proof}

In view of the equivalence in Lemma~\ref{lem:ortho_equiv_Pars}, Proposition~\ref{prop:discr_orth} can be reformulated
to give a characterization, when a collection of vectors $\{\bvarphi_x\}_{x\in\Omega}$ in $\K^n$ can be made a Parseval $\mu$-frame by choice of a suitable measure $\mu$.
This reformulation is the central result of this section. It also applies to general Hilbert spaces of finite dimension since every $n$-dimensional Hilbert space $\cH$ 
is isomorphic to $\K^n$.

\begin{theorem}\label{thm:discr_orth} %%%%%%%%%%%%%%%%%%%%%%%%%%%%%%%%%%%
    Let $\{\bvarphi_x\}_{x\in\Omega}$ be a family of vectors in $\K^n$, without any specific assumptions on the index set $\Omega$, and let $\bvarphi:\Omega\to\K^n$ be the function defined by assigning $x\mapsto\bvarphi_x$. 
    The following assertions are equivalent.
    \begin{enumerate}[label=\normalfont(\roman*)]
    \item
        We have
        \begin{equation}\label{eq:cone-condition}
            \bI_n \in \cone \Big\{
            \bvarphi_x\bvarphi_x^\ast
            \colon\,x\in \Omega \Big\}\,.
        \end{equation}
     \item
        There exists a positive measure $\mu$ on $\Omega$ such that $\bvarphi\in L_2(\Omega,\mu;\K^n)$ and $\bvarphi$ is a Parseval $\mu$-frame for $\K^n$.
    \item
        There exists $N \le \dim_{\K}\lin_{\K}\{\bvarphi_x\bvarphi_x^\ast \colon\,x\in \Omega\}$ and points $x_1,\ldots,x_{N}\in \Omega$ and weights $\mu_1,\ldots,\mu_{N}>0$ such that
            $\bvarphi$ is a Parseval $\widehat{\mu}$-frame with respect to the discrete measure $\widehat{\mu}=\sum_{j=1}^{N} \mu_j \delta_{x_j}$.
    \end{enumerate}
\end{theorem} 
\begin{proof}
Let $\varphi_1,\ldots,\varphi_n:\Omega\to\K$ be the component functions of $\bvarphi$
and observe $\bvarphi_x\bvarphi_x^\ast=(\varphi_k(x)\overline{\varphi_l(x)})_{k,l=1}^{n}$ for every $x\in\Omega$. Hence, \eqref{eq:cone-condition} is equivalent to condition (i) in Proposition~\ref{prop:discr_orth}. Further, due to~\eqref{eq:dim_equal}, 
\begin{align*}
\dim_{\K} \lin_{\K}\{\bvarphi_x\bvarphi_x^\ast :x\in\Omega \}=\dim_{\K}\lin_{\K}\{\varphi_k\overline{\varphi_l}:k,l=1,\ldots,n \}  \,.
\end{align*}
The rest follows from Lemma~\ref{lem:ortho_equiv_Pars}.
The orthonormality condition~\eqref{eq:ortho_cond} corresponds to (ii), and
(iii) here correlates as well to (iii) in Proposition~\ref{prop:discr_orth}.
\end{proof}

Theorem~\ref{thm:discr_orth}, together with Remark~\ref{rem:discr_orth_2}, should be compared with~\cite[Thm.~3.1]{BKPSU24}, noting that here we are in a more general setting. A further extension of Theorem~\ref{thm:discr_orth} is given in Theorem~\ref{thm:discr_orth_2} at the end of this section. 
First, we proceed with a number of useful corollaries for $\mu$-frames.
For their formulation, we need to assume that the index set $\Omega$ in Theorem~\ref{thm:discr_orth} is a positive measure space $(\Omega,\Sigma,\mu)$. 

The first corollary characterizes discrete scalability. It should be compared with results in~\cite{KUTYNIOK20132225} and~\cite{Kutyniok_Okoudjou_Philipp}. The proof is a straight-forward application of Theorem~\ref{thm:discr_orth}.

\begin{corollary}\label{cor:discr_scalability}
A $\mu$-frame $\bvarphi:\Omega\to\K^n$ for $\K^n$ is discretely scalable if and only if one of the conditions {\rm(i)-(iii)} of Theorem~\ref{thm:discr_orth} holds true for the family $\{\bvarphi_x\}_{x\in\Omega}$, where $\bvarphi_x:=\bvarphi(x)$.
\end{corollary}

Another result concerns the discretizability of the frame operator.

\begin{corollary}\label{cor:discr_orth}
Let $\bvarphi:\Omega\to\K^n$ be a
\emph{$\mu$-frame} for $\K^n$ and $S_{\bvarphi}:=S_{\bvarphi}(\mu)$ the associated frame operator.
Then there exist points $x_1,\ldots,x_N$, with $N$ as in Theorem~\ref{thm:discr_orth}, and weights $\mu_1,\ldots,\mu_N>0$ such that
\begin{align*}
S_{\bvarphi} = \sum_{j=1}^{N} \mu_j \bvarphi(x_j) \bvarphi(x_j)^\ast \,.  
\end{align*}
\end{corollary}
\begin{proof}
Every $\mu$-frame $\bvarphi:\Omega\to\K^n$ possesses 
a canonical isomorphic Parseval $\mu$-frame $\bpsi:\Omega\to\K^n$, namely $\bpsi(x):=S_{\bvarphi}^{-1/2}\bvarphi(x)$.  
By the equivalence of the statements (ii) and (iii) in Theorem~\ref{thm:discr_orth}, we get
\begin{align*}
\bI_n =  S_{\bpsi}(\widehat{\mu}) = \sum_{j=1}^{N} \mu_j \bpsi(x_j) \bpsi(x_j)^\ast
\end{align*}
for certain points $x_1,\ldots,x_N$,, and certain weights $\mu_1,\ldots,\mu_N>0$,  where $N\le \dim_{\K}\lin_{\K}\{\bvarphi_x\bvarphi_x^\ast \colon\,x\in \Omega\}$ as in Theorem~\ref{thm:discr_orth}. This implies, due to $S^\ast_{\bvarphi}=S_{\bvarphi}$,
\begin{gather*}
S_{\bvarphi} = \sum_{j=1}^{N} \mu_j S^{1/2}_{\bvarphi}\bpsi(x_j) \bpsi(x_j)^\ast S^{1/2}_{\bvarphi}   
=  \sum_{j=1}^{N} \mu_j S^{1/2}_{\bvarphi}\bpsi(x_j) \big(S^{1/2}_{\bvarphi}\bpsi(x_j)\big)^\ast = \sum_{j=1}^{N} \mu_j \bvarphi(x_j) \bvarphi(x_j)^\ast \,. \qedhere
\end{gather*}
\end{proof}

As the sum in Corollary~\ref{cor:discr_orth} reproduces the frame operator $S_{\bvarphi}$ exactly, we have in particular, due to~\eqref{eq:spectral_frameop}, 
\begin{align*}
A \|f\|^2 \le \sum_{j=1}^{N} \mu_j |\langle f,\bvarphi(x_j) \rangle|^2  \le B \|f\|^2 \quad\text{ for all }f\in\cH \,,
\end{align*}
if $A$ and $B$ are the frame bounds of $\bvarphi$. 
We hence also obtain the following result.

\begin{corollary}\label{cor:discr_subframe}
Let $\bvarphi:\Omega\to\K^n$ be a
\emph{$\mu$-frame} for $\K^n$ with frame bounds $A$,$B$.
Then there exist points $x_1,\ldots,x_N\in\Omega$, with $N$ as in Theorem~\ref{thm:discr_orth}, and weights $s_1,\ldots,s_N>0$ such that
\begin{align*}
   \bvarphi_{\rm discr}:= \big\{s_j\bvarphi(x_j)\big\}_{j=1}^{N}
\end{align*} 
is a discrete frame for $\K^n$ with the same frame bounds.
\end{corollary}

This last statement is a positive answer to the problem of exact weighted frame subsampling in finite-dimensional Hilbert spaces. In contrast, if no reweighting of the subsampled frame elements is allowed, it is not always possible to preserve the frame bounds. A simple example in $\K^2$ is given below.

\begin{example}\label{ex:unweight_sub}
Let $e_1=(1,0)^\top$, $e_2=(0,1)^\top$ denote the canonical unit vectors in $\K^2$.
For a sequence $\{\lambda_j\}_{j\in\N}\subset\R_{>0}$ of strictly positive numbers with $\sum_{j\in\N} \lambda_j^2 = 1$ the collection $\{f_j\}_{j\in\N_0}$ of vectors $f_0:=e_1$ and $f_j:=\lambda_j e_2$, $j\in\N$,
is then a discrete Parseval frame in $\K^2$. But, for any subsampled frame $\{f_j\}_{j\in J_0}$ with $J_0:=\{0\}\cup J$, where $J\subsetneq\N$ is a non-empty proper subset of $\N$, we have the frame bounds $\sum_{j\in J} \lambda_j^2 = A<B=1$.
\end{example}

Unweighted frame subsampling, i.e., subsampling without rescaling, is a much more difficult problem. As demonstrated by Example~\ref{ex:unweight_sub}, exactness can usually not be achieved. But, for some parameter $\varepsilon>0$, one can aim for an approximate discretization, i.e., for a subsampled frame with comparable frame bounds, usually of the form $(1-\varepsilon)A$ and $(1+\varepsilon)B$. Such results have been obtained on the basis of Weaver's $\textup{KS}_2$-theorem, whose statement is equivalent to the Kadison-Singer theorem, see e.g.~\cite{CasazzaEdidin2007}, and was famously proved by Marcus, Spielman, and Srivastava~\cite{MaSpSr15} in 2015, and ideas from~\cite{NiOlUl16}, see e.g.~\cite[Lem.~2.2]{LimTem22},~\cite[Thm.~2.3]{NaSchUl22}, or \cite[Prop.~17]{DoKrUl23}. These latter results all apply to discrete frames.
Concerning continuous frames, let us mention
the seminal work~\cite{FS19} by Freeman and Speegle. It shows 
that every bounded continuous frame $\bvarphi:\Omega\to\cH$, i.e., one for which $\sup_{x\in\Omega} \|\bvarphi(x)\|_{\cH}<\infty$, 
contains a countable subframe with comparable bounds for separable $\cH$. Other articles related to this subject are e.g.~\cite{BodmannCasazzaPaulsenSpeegle2010,Bownik2017_Lyapunov}.

We end this section with Theorem~\ref{thm:discr_orth_2} below, which summarizes most of our investigations in this section since it contains the statements of Theorem~\ref{thm:discr_orth} and Corollary~\ref{cor:discr_orth} as special cases.

\begin{theorem}\label{thm:discr_orth_2} %%%%%%%%%%%%%%%%%%%%%%%%%%%%%%%%%%%
    Let $\{\bvarphi_x\}_{x\in\Omega}$ be a family of vectors in $\K^n$, with arbitrary index set $\Omega$, and let the function $\bvarphi:\Omega\to\K^n$ be defined by assigning $x\mapsto\bvarphi_x$. 
    For a positive definite Hermitian matrix $\bF\in\K^{n\times n}$,
    the following assertions are equivalent.
    \begin{enumerate}[label=\normalfont(\roman*)]
    \item
        We have
        \begin{equation}\label{eq:cone-condition_2}
            \bF \in \cone \Big\{
            \bvarphi_x\bvarphi_x^\ast
            \colon\,x\in \Omega \Big\}\,.
        \end{equation}
     \item
        There exists a positive measure $\mu$ on $\Omega$ such that $\bvarphi$ is a $\mu$-frame for $\K^n$ with frame operator $S_{\bvarphi}(\mu)=\bF$.
    \item
        There exists $N \le \dim_{\K}\lin_{\K}\{\bvarphi_x\bvarphi_x^\ast \colon\,x\in \Omega\}$ and points $x_1,\ldots,x_{N}\in \Omega$ and weights $\mu_1,\ldots,\mu_{N}>0$ such that
            $\bvarphi$ is a $\widehat{\mu}$-frame with respect to the discrete measure $\widehat{\mu}=\sum_{j=1}^{N} \mu_j \delta_{x_j}$ and $S_{\bvarphi}(\widehat{\mu})=\bF$.
    \end{enumerate}
\end{theorem} 
\begin{proof}
Define $\bm\tilde{\bvarphi}_x:=\bF^{-1/2}\bvarphi_x$. Condition (i) is then equivalent to $\bI_n \in \cone \big\{
            \bm\tilde{\bvarphi}_x\bm\tilde{\bvarphi}_x^\ast
            \colon\,x\in \Omega \big\}$. The statement now follows from applying Theorem~\ref{thm:discr_orth} to the family $\{\bm\tilde{\bvarphi}_x\}_{x\in\Omega}$.
\end{proof}

Theorem~\ref{thm:discr_orth_2} draws the boundaries for the possibility of `tuning a $\mu$-frame' $\bvarphi$, in the sense of improving its frame bounds, via a change of measure. According to the result, this is possible if and only if
there is a positive definite matrix $\bF$ better conditioned than $S_{\bvarphi}(\mu)$ and satisfying the cone condition~\eqref{eq:cone-condition_2}.
Theorem~\ref{thm:discr_orth_2} then asserts 
the existence of a discrete measure $\widehat{\mu}=\sum_{j=1}^{N} \mu_j \delta_{x_j}$ with
$N\le \dim_{\K}\lin_{\K}\{\bvarphi_x\bvarphi_x^\ast \colon\,x\in \Omega\}$ such that $\bvarphi$ is a $\widehat{\mu}$-frame with the better conditioned frame operator $S_{\bvarphi}(\widehat{\mu})=\bF$. 
If $\bvarphi$ is discretely scalable, one can thus aim for $S_{\bvarphi}(\widehat{\mu})=\bI_n$. If this is not the case, one can still try to find reweighted subframes with better condition. Not in all cases, however, an optimal subframe exists. 
Investigations on `frame tuning' is an ongoing effort in frame theory, see  e.g.~\cite{CasaCarliTran23}.

\begin{remark}\label{rem:discr_orth_2}
As for Proposition~\ref{prop:discr_orth}, see~Remark~\ref{rem:discr_orth}, convex versions of Theorems~\ref{thm:discr_orth} and~\ref{thm:discr_orth_2} can be stated. 
If in item (i) the convex hull 
\[
\conv \Big\{
            \bvarphi_x\bvarphi_x^\ast
            \colon\,x\in \Omega \Big\}
\]
is used instead of the conic hull
the measure $\mu$ in (ii) can be realized as a probability measure and in (iii), allowing
$N+1$ points and weights, the additional condition $\mu_1 + \ldots + \mu_{N+1}=1$ can be fulfilled.
\end{remark}

%%%%%%%%%%%%%%%%%%%%%%%%%%%%%%%%%%%%%%%%%%%%%%%%%%%%
%--------------------------------------------------%
%%%%%%%%%%%%%%%%%%%%%%%%%%%%%%%%%%%%%%%%%%%%%%%%%%%%

%\newpage

\section{\texorpdfstring{$D$-optimal}{D-optimal} design}
\label{sec:$D$-opt}

For practical applications it is desirable that the point-weight pairs $\{(x_j,\mu_j)\}_{j=1}^{N}\subset\Omega\times\R_{>0}$ in item (iii) of Proposition~\ref{prop:discr_orth}, Theorem~\ref{thm:discr_orth}, or Theorem~\ref{thm:discr_orth_2}  
can be determined algorithmically.
Therefore we present in this last section a more constructive approach for obtaining those.
The basic idea is to formulate the task as an extremal problem with an objective function that has to be maximized. It is a strategy reminiscent of the utilization of frame potentials for the construction of well-conditioned frames.

\begin{remark}
Constructivity in our context here just means that the solution procedure is guided by an underlying extremal principle. In this sense the proposed $D$-design method is constructive. However, it usually involves the solution of a difficult non-convex optimization problem.
Example~\ref{ex:Ddesign}, at the end of this section, is an illustration, where the $D$-design approach leads to the problem of finding the single solution $x_0\in\N$ of $\varphi(x)=1$ for a function $\varphi:\N\to\{0,1\}$ with $|\{x\in\N : \varphi(x)=1\}|=1$.
Efficient computational solutions can hence only exist for suitably restricted subclasses of the problem. 
\end{remark}

In our analysis it suffices to consider the problem of determining the point-weight pairs in Theorem~\ref{thm:discr_orth} (iii), since Proposition~\ref{prop:discr_orth} is just a reformulation of Theorem~\ref{thm:discr_orth} and the problem of finding the point-weight pairs in Theorem~\ref{thm:discr_orth_2} (iii)
from data $\bF\in\K^{n\times n}$ and $\bm\tilde{\bvarphi}:\Omega\to\K^n$ can be reduced to
finding a point-weight set $\{(x_j,\mu_j)\}_{j=1}^{N}\subset\Omega\times\R_{>0}$ satisfying
\begin{align}\label{eq:aim_Ddesign_2}
\bI_n = \sum_{j=1}^{N} \mu_j \bvarphi(x_j) \bvarphi(x_j)^\ast   
\end{align}
for the transformed function
$\bvarphi(x):=\bF^{-1/2}\bm\tilde{\bvarphi}(x)$.
Recall that the input 
$\bF\in\K^{n\times n}$ is required to be a positive definite Hermitian matrix that satisfies the cone condition~\eqref{eq:cone-condition_2}.
The new function $\bvarphi:\Omega\to\K^n$ then clearly fulfills
the cone condition~\eqref{eq:cone-condition} in Theorem~\ref{thm:discr_orth}, namely
\begin{align}\label{eq:cone-condition-Dopt}
\bI_n \in %\cC_{\bvarphi} :=
\cone \Big\{
            \bvarphi(x)\bvarphi(x)^\ast
            \colon\,x\in \Omega \Big\} =: \cC_{\bvarphi}
            \,.
\end{align}
The essential task is hence to find points and weights for the representation~\eqref{eq:aim_Ddesign_2} under the assumption~\eqref{eq:cone-condition-Dopt}.
To solve this problem, we subsequently 
present an idea firstly elaborated in~\cite[Sec.~3]{BKPSU24}.

A central element is the \emph{Christoffel function} \begin{align*}
\gamma_{n}(x) 
:= \frac{1}{n} \|\bvarphi(x)\|_2^2  
\end{align*}
associated to $\bvarphi$. With this function, we can
make the transformation 
\begin{align}\label{rescaling_of_varphi}
    \bpsi(x) :=  \bvarphi(x)\sqrt{\omega_{n}(x)} \quad\text{ with}\quad \omega_{n}(x) :=  \begin{cases} 1/\gamma_{n}(x)  &,\, \gamma_n(x)\neq 0 \,, \\
    0 &,\, \gamma_n(x) = 0 \,.
    \end{cases}
\end{align}

For each fixed $x\in\Omega$, we then obviously have 
\begin{align}\label{sum_transformed}
\|\bpsi(x)\|_2^2 
= \|\bvarphi(x)\|_2^2 \cdot\omega_{n}(x) = \begin{cases} n  &,\, \omega_n(x)\neq 0 \,, \\
    0 &,\, \omega_n(x) = 0 \,.
    \end{cases}
\end{align}
Also, the following equality holds true,
\begin{align*}
\cC_{\bvarphi} 
            = \cone \Big\{
            \bpsi(x)\bpsi(x)^\ast
            \colon\,x\in \Omega \Big\} =: \cC_{\bpsi} \,.
\end{align*}
Let us further define the convex hull $\Delta_{\bpsi}$ generated by $\bpsi(x)\bpsi(x)^\ast$, i.e.,
\begin{align*}
\Delta_{\bpsi} := \conv \Big\{
            \bpsi(x)\bpsi(x)^\ast
            \colon\,x\in \Omega \Big\}\,.
\end{align*}

\begin{lemma}\label{lem:D-optimal_1}
It holds $\bI_n \in \cC_{\bvarphi} \Leftrightarrow
            \bI_n \in \Delta_{\bpsi}$.
Further, every matrix $\bA\in\Delta_{\bpsi}$ satisfies $\Tr \big(\bA\big)
        \leq n$.
\end{lemma}
\begin{proof}
$\bI_n \in \cC_{\bvarphi}$ means that there exist $M\in\N$, points $x_1,\ldots,x_M\in\Omega$, and non-negative weights $\beta_1,\ldots,\beta_M$ such that
\begin{align}\label{abcxyz}
\bI_n =  \sum_{j=1}^{M}\beta_j  \bvarphi(x_j)\bvarphi(x_j)^\ast \,,
\end{align}
which can be rewritten as 
\begin{align}\label{xyzabc}
\bI_n =  \sum_{j=1}^{M}\alpha_j \bpsi(x_j)\bpsi(x_j)^\ast \quad\text{ with weights $\alpha_j:=\beta_j\cdot\gamma_{n}(x_j)$}\,.
\end{align}
Further, due to~\eqref{abcxyz}, it holds
\begin{align*}
\sum_{j=1}^{M} \alpha_j =
\sum_{j=1}^{M} \beta_j\gamma_{n}(x_j) = \frac{1}{n} \sum_{j=1}^{M} \beta_j \|\bvarphi(x_j)\|_2^2 = \frac{1}{n} \sum_{j=1}^{M} \beta_j \Tr(\bvarphi(x_j)\bvarphi(x_j)^\ast) = 1 \,,
\end{align*}
proving $\bI_n \in \Delta_{\bpsi}$. For the opposite direction, we start with~\eqref{xyzabc} and non-negative weights $\alpha_j$ summing up to $1$. Then clearly~\eqref{abcxyz} holds with $\beta_j:=\alpha_j\cdot\omega_{n}(x_j)$ and thus $\bI_n \in \cC_{\bvarphi}$.
Next, let $\bA\in\Delta_{\bpsi}$ with representation $\sum_{j=1}^{M}\alpha_j  \bpsi(x_j)\bpsi(x_j)^\ast$. Then, in view of \eqref{sum_transformed},
    \begin{gather*}
        \Tr \Big(\sum_{j=1}^{M}\alpha_j  \bpsi(x_j)\bpsi(x_j)^\ast\Big)
        =  \sum_{j=1}^{M}\alpha_j  \Tr(\bpsi(x_j)\bpsi(x_j)^\ast)
        = \sum_{j=1}^{M}\alpha_j
         \|\bpsi(x_j)\|_2^2 
        \le \sum_{j=1}^{M}\alpha_j n = n \,. \qedhere
    \end{gather*}
\end{proof}

Lemma~\ref{lem:D-optimal_1} provides an estimate for the trace $\Tr(\bA)$ of matrices $\bA\in\Delta_{\bpsi}$. The subsequent lemma, similarly, provides an estimate for the determinant $\det(\bA)$.

\begin{lemma}\label{lem:D-optimal_2}
Let $\bA\in\Delta_{\bpsi}$. Then $\det(\bA)\le 1$ and
$\det\big(\bA\big) = 1 \Leftrightarrow
            \bI_n = \bA$.
\end{lemma}
\begin{proof}
Due to Lemma~\ref{lem:D-optimal_1} and the inequality between arithmetic and geometric mean, it holds 
    \begin{equation}\label{aux:equation}
        \sqrt[n]{\det\big(\bA
        \big)}
        \le \frac{1}{n}
        \Tr \big(\bA\big)
        \leq 1\,.
    \end{equation}
This proves the first assertion. Since $\det(\bI_n)=1$, it only remains to prove $\bA=\bI_n$ for $\bA\in\Delta_{\bpsi}$ with $\det(\bA)=1$. But due to~\eqref{aux:equation}, the assumption $\det(\bA)=1$ necessarily implies $\Tr(\bA)=n$.
Therefore both, the geometric and the arithmetic mean of the eigenvalues of $\bA$, are $1$, which is only possible if all eigenvalues are $1$, whence $\bA = \bI_{n}$ since $\bI_{n}$ is the only Hermitian matrix with that property.
\end{proof}

Based on Lemma~\ref{lem:D-optimal_2}, we next transfer the maximization   
procedure from~\cite[Sec.~3]{BKPSU24} to our generalized setting.
For this we introduce, for each $N\in\N$, the class 
\begin{align*}
\cK_N := \bigg\{ \kappa= \big\{\big(x^{(\kappa)}_j,\alpha^{(\kappa)}_j\big)\big\}_{j=1}^{N} \colon\, x^{(\kappa)}_1,\ldots,x^{(\kappa)}_{N}\in\Omega\,;\: \alpha^{(\kappa)}_1,\ldots,\alpha^{(\kappa)}_{N}\in\R_{>0}\,;\: \sum_{j=1}^{N} \alpha^{(\kappa)}_j=1  \bigg\} 
\end{align*}
of weighted knots
and, for each $\kappa= \{(x^{(\kappa)}_j,\alpha^{(\kappa)}_j)\}_{j=1}^{N}\in\cK_N$,
we let $\bA_{\kappa}$ denote the matrix 
\[
\bA_{\kappa} := \sum_{j=1}^{N} \alpha^{(\kappa)}_j \bpsi(x^{(\kappa)}_j) \bpsi(x^{(\kappa)}_j)^\ast \,.
\]
Due to Lemma~\ref{CarSub}~(ii), we get for 
$N\ge \dim_{\R}\lin_{\R}\{\bpsi(x)\bpsi(x)^\ast \colon\,x\in \Omega\}+1$
\begin{align}\label{eq:knots_restriction}
  \big\{ \bA_{\kappa} : \kappa\in\cK_{N} \big\} = \conv \big\{ \bpsi(x)\bpsi(x)^\ast
            \colon\,x\in \Omega \big\} = \Delta_{\bpsi} \,. 
\end{align}
Hence, if $N$ is sufficiently large and $\bI_n\in \Delta_{\bpsi}$, there is a node and weight set $\kappa^\ast\in\cK_{N}$ that fulfills
\begin{align*}
\bI_n = \bA_{\kappa^\ast} \quad\text{and}\quad \det(\bA_{\kappa^\ast}) = 1 = \max\limits_{\bA \in \Delta_{\bpsi}} \det(\bA) \,.
\end{align*}
This motivates the following maximization procedure, which follows the principle of $D$-optimal design,
\begin{equation}\label{det_pursuit}
  \kappa^\ast = \argmax\limits_{\kappa\in\cK_N} \det(\bA_{\kappa})  \,,
\end{equation}
whereby $N$ needs to be chosen large enough.
This method has already been presented in~\cite[Thm.~3.1]{BKPSU24} for compact domains $\Omega$ and continuous $\bvarphi:\Omega\to\K^n$.
According to the next theorem, it
is indeed well-defined
whenever the necessary condition $\bI_n \in \Delta_{\bpsi}$ is fulfilled and $N> \dim_{\R}\lin_{\R}\{\bpsi(x)\bpsi(x)^\ast \colon\,x\in \Omega\}$. 
No compactness assumption as in~\cite[Thm.~3.1]{BKPSU24} is needed.
Further, by Lemma~\ref{lem:D-optimal_1},
the natural assumption $\bI_n \in \Delta_{\bpsi}$
is equivalent to $\bI_n \in \cC_{\bvarphi}$, the 
criterion for discrete scalability of $\bvarphi$.

\begin{theorem}%[$D$-optimal design]
\label{$D$-opt} %%%%%%%%%%%%%%%%%%%%%%%%%%%%%%%%%%%
    Let $\Omega$ be a set and $\bvarphi:\Omega\to\K^n$
    a function. Let further 
    $\bI_n\in\K^{n\times n}$ be
    the identity matrix.
    The following assertions are equivalent.
    \begin{enumerate}[label=\normalfont(\roman*)]
    \item
        We have $\bI_n\in\cC_{\bvarphi}$.
    \item
        We have $\bI_n\in\Delta_{\bpsi}$ with $\bpsi$ as in~\eqref{rescaling_of_varphi}.
    \item 
    The matrix $\bI_n$ is the unique maximizer $\bA_{\kappa^\ast}$ of the determinant maximization problem~\eqref{det_pursuit} %for $N\in\N$ 
    if $N\ge \dim_{\K}\lin_{\K}\{\bvarphi(x)\bvarphi(x)^\ast \colon\,x\in \Omega\} +1$. 
    \end{enumerate}
\end{theorem} 
\begin{proof}
Due to Lemma~\ref{lem:D-optimal_1}, assumptions (i) and (ii) are equivalent. Further, if those are fulfilled, 
by Lemma~\ref{lem:D-optimal_2}, the identity matrix $\bI_n$ is the only matrix in $\Delta_{\bpsi}$ with determinant $1$ and all other matrices in $\Delta_{\bpsi}$ have determinant strictly less than $1$. The identity $\bI_n$ is thus the only maximizer of~\eqref{det_pursuit} if $N\in\N$ is chosen larger than $\dim_{\R}\lin_{\R}\{\bpsi(x)\bpsi(x)^\ast \colon\,x\in \Omega\}$. In view of~\eqref{eq:knots_restriction}, such a choice of $N$ guarantees the existence of $\kappa^\ast\in\cK_N$ with $\bA_{\kappa^\ast}=\bI_n$. The statement (iii) follows, since $\dim_{\K}\lin_{\K}\{\bvarphi(x)\bvarphi(x)^\ast \colon\,x\in \Omega\} = \dim_{\K}\lin_{\K}\{\bpsi(x)\bpsi(x)^\ast \colon\,x\in \Omega\}$ and $\dim_{\K}\lin_{\K}\{\bpsi(x)\bpsi(x)^\ast \colon\,x\in \Omega\} = \dim_{\R}\lin_{\R}\{\bpsi(x)\bpsi(x)^\ast \colon\,x\in \Omega\}$, the last equality due to~\cite[Lem.~2.6]{BKPSU24}.
\end{proof}

If a maximizer $\bA_{\kappa^\ast} = \sum_{j=1}^{N} \alpha_j \bpsi(x_j)\bpsi(x_j)^\ast$ in~\eqref{det_pursuit} has been determined, we have $\bI_n = \bA_{\kappa^\ast}$ as desired, according to Theorem~\ref{$D$-opt}, and
finally get the representation \eqref{eq:aim_Ddesign_2}
for the nodes $x_1,\ldots,x_N$ and the weights
\[
\mu_j:= \alpha_j\cdot \omega_n(x_j) \,.
\]
Note however that, while the 
maximizing matrix $\bA_{\kappa^\ast}=\bI_n$ is unique as a consequence of Lemma~\ref{lem:D-optimal_2}, the maximizing node set $\kappa^\ast$ is usually not unique.

\begin{example}\label{ex:Ddesign}
Consider $V_1:=\lin_{\R}\{\varphi_1\}$ 
for a function $\varphi_1:\N\to\{0,1\}$ which vanishes everywhere apart from one fixed $x_0\in\N$.
In this example $n=1$, $N=2$, and $\omega_1=\gamma_1=\varphi_1$. We hence have $\bpsi=\bvarphi$ and $\cC_{\bvarphi}=\cC_{\bpsi}=[0,\infty)$, $\Delta_{\bpsi}=[0,1]$. The $D$-design~\eqref{det_pursuit},
\[
\kappa^\ast = \argmax\limits_{\kappa\in\cK_2} \det(\bA_{\kappa}) = \argmax\limits_{\{(x_1,\alpha_1),(x_2,\alpha_2)\}\in\cK_2} \alpha_1\varphi_1(x_1)+ \alpha_2\varphi_1(x_2) \,,
\]
is here equivalent to finding the only point $x_0\in\N$ with $\varphi_1(x_0)=1$.
\end{example}

%%%%%%%%%%%%%%%%%%%%%%%%%%%%%%%%%%%%%%%%%%%%%%%%%%%%
%--------------------------------------------------%
%%%%%%%%%%%%%%%%%%%%%%%%%%%%%%%%%%%%%%%%%%%%%%%%%%%%

\section*{Acknowledgement}
The authors would like to thank Erich Novak 
for 
his informative remarks on the topic of optimality of node sets
and for pointing out~\cite{NoHi07}. 
%... the anonymous referees for carefully proof-reading the manuscript ...

%\bibliographystyle{abbrv}
%\bibliography{references.bib}

\end{document}